\documentclass[12 pt,a4paper,leqno]{article}

\usepackage[french,english]{babel}
\usepackage{latexsym,amsmath,amscd,amsfonts,amssymb,mathrsfs,amsthm}
\usepackage{graphicx}
\usepackage[all]{xy}
\usepackage[T1]{fontenc}

\RequirePackage{amsmath}
\RequirePackage{amssymb}
\RequirePackage{ifthen}
\RequirePackage{pslatex}

\usepackage[super]{nth}

\newtheorem{thm}{Theorem}

\newtheorem{lem}{Lemma}
\newtheorem{prop}{Proposition}
\newtheorem{cor}{Corollary}
\newtheorem{rem}{Remark}

\newtheorem*{ent}{Entry}
\newtheorem*{fact}{Fact}

\author{
  Jacques Sauloy\footnote{E-mail address:jacques.sauloy@gmail.com,
    URL: http://www.cantoperdic.fr/},
    Changgui Zhang\footnote{Laboratoire P. Painlev\'e (UMR-CNRS 8524), UFR Math.,
    Universit\'e de Lille 1, Cit\'e scientifique,
    59655 Villeneuve d'Ascq cedex, France,
    E-mail address:changgui.zhang@univ-lille.fr,
    URL:https://math.univ-lille1.fr/~zhang/}}

\title{On the vanishing of coefficients of the powers of a theta function}

\def\C{{\mathbf C}}
\def\Q{{\mathbf Q}}
\def\Z{{\mathbf Z}}
\def\R{{\mathbf R}}
\def\N{{\mathbf N}}

\def\Cs{{\mathbf C^*}}

\def\Im{\text{Im}}

\def\lmod{\left |}            
\def\rmod{\right |}           
\def\tq{~ | ~}              
\def\ie{\emph{i.e.}}

\def\1{\underline{1}}         

\def\ii{{\text{i}}}

\def\Raq(d){{\C\{\xi\}_{q,(\delta)}}}
\def\Kaq(d){{\C(\{\xi\})_{q,(\delta)}}}

\def\thq{{\theta_q}}

\def\card{{\text{card}}}

\def\Do{{\overset{\circ}{\mathbf{D}}}}
\def\Df{{\overline{\mathbf{D}}}}

\def\m{{\underline{m}}}
\def\n{{\underline{n}}}
\def\x{{\underline{x}}}
\def\y{{\underline{y}}}
\def\u{{\underline{u}}}
\def\v{{\underline{v}}}
\def\b{{\underline{b}}}
\def\es{{\underline{\epsilon}}}
\def\tr{{}^t\!}

\newcommand {\cB}{\mathcal{B}} 
\newcommand {\cT}{\mathcal{T}} 

\begin{document}

\selectlanguage{french}

\maketitle


\selectlanguage{english}

\begin{abstract}
  A result on the Galois theory of $q$-difference equations \cite{JSTALPAEN}
  leads to the following question: if $q \in \Cs$, $\lmod q \rmod < 1$ and if
  one   sets $\thq(z) := \sum\limits_{m \in \Z} q^{m(m-1)/2} z^m$, can some coefficients
  of the Laurent series expansion of $\theta_q^k(z)$, $k \in \N^*$, vanish ?
  We give a partial answer.
\end{abstract}



\section{Introduction}
\label{section:introduction}


\subsection{Origin of the problem}

Let $q \in \C$ such that $\lmod q \rmod > 1$. In \cite{JSTALPAEN}, one uses
the function\footnote{Note however that the conventions in the present work 
are different, one assumes that $0 < \lmod q \rmod < 1$ and it is the formula
\eqref{eqn:deftheta} herebelow which defines $\thq$. Also, notation for the
coefficients will differ, see formula \eqref{eqn:formuleexactecoeffstheta}.}
$\thq(z) := \sum\limits_{m \in \Z} q^{-m(m+1)/2} z^m$, $z \in \Cs$.
This is a holomorphic function over $\Cs$. For every $k \in \N^*$, one introduces
the coefficients $t_n^{(k)}$, $n \in \Z$, through the expansion into a Laurent
series $\theta_q^k(z) = \sum t_n^{(k)} z^n$. These series converge normally over
every compact subset of the domain $\C \setminus \Df(0,1)$, whence the explicit
formula:
$$
t_n^{(k)} =
\sum_{m_1,\ldots,m_k \in \Z \atop m_1 + \cdots + m_k = n} q^{- (m_1(m_1+1) + \cdots + m_k(m_k+1))/2},
$$
which implies that each $t_n^{(k)}$ is,as a function of $q$, holomorphic over the
domain $\C \setminus \Df(0,1)$. Since $q > 0 \Rightarrow t_n^{(k)}(q) > 0$, the
function $t_n^{(k)}$ is not identically $0$, so the set of its zeroes is discrete,
whence denumerable. As a consequence, the set of all zeroes of all functions
$t_n^{(k)}$ is denumerable. \\

For all $q$ such that no $t_n^{(k)}(q)$ vanishes, the first author has described
in \cite{JSTALPAEN} the ``wild fundamental group'' of analytic complex linear
$q$-difference equations which are irregular at $0$ (and with arbitrary slopes).
This led to call ``good values'' of $q$ such values. As we just saw, this property
is generically satisfied. The goal of the present article is to study various cases
of vanishing of $t_n^{(k)}(q)$. \\

However, to stick as much as possible to the classical litterature on $q$-special
functions (above all Ramanujan's opus !), we shall assume in this work that
$\lmod q \rmod < 1$, $q \neq 0$, and one shall begin with the following definition:
\begin{equation}
\label{eqn:deftheta}
\thq(x) := \sum_{m \in \Z} q^{m(m-1)/2} x^m, \quad x \in \Cs,
\end{equation}
which again defines a holomorphic function over $\Cs$. For every $k \in \N^*$, we
shall denote $\gamma_{k,n}$, $n \in \Z$ the coefficients of the Laurent series expansion
of $\theta_q^k$:
\begin{equation}
\label{eqn:formuleexactecoeffstheta}
\theta_q^k(q) = \sum_{n \in \Z} \gamma_{k,n} q^n \Longrightarrow
\gamma_{k,n}(q) = \gamma_{k,n} = 
\sum_{m_1,\ldots,m_k \in \Z \atop m_1 + \cdots + m_k = n} q^{(m_1(m_1-1) + \cdots + m_k(m_k-1))/2},
\end{equation}

The relation to the formulas in \cite{JSTALPAEN} is clear: if $\lmod q \rmod > 1$
and if $p := q^{-1}$, then $\thq(z) = \theta_p(x)$ with $x := p z$, and moreover
$t_n^{(k)}(q) = \gamma_{k,n}(q^{-1})$; and the ``good values'' of $q$, $\lmod q \rmod > 1$,
are those such that none of the $\gamma_{k,n}(q^{-1})$ vanishes. \\

In this work, we shall first study $\thq$, $\lmod q \rmod < 1$, $q \neq 0$, and the
coefficients $\gamma_{k,n}$. However, in \ref{subsection:Puissance4} and then again
in \ref{section:-1} we shall rather use (for simplicity of the computations) another
of the classical theta functions\footnote{This definition requires an unambiguous
determination of $q^{1/2}$, the necessary conventions are explained in
\ref{subsubsection:conventionsPuissance4}.}:
\begin{equation}
\label{eqn:defvartheta}
\vartheta(x) := \vartheta_q(x) := \sum\limits_{n\in\Z} q^{n^2/2} x^n.
\end{equation}
We correspondingly introduce the coefficients of powers of $\vartheta$:
\begin{equation}
\label{eqn:formuleexactecoeffsvartheta}
\vartheta_q^k(z) = \sum_{n \in \Z} c_{k,n} z^n \Longrightarrow
c_{k,n}(q) = c_{k,n} =
\sum_{(n_1,...,n_k) \in \Z^k \atop n_1 + \cdots + n_k = n} q^{(n_1^2 + \cdots + n_k^2)/2}.
\end{equation}

From the obvious relations:
$$
\vartheta(x) = \thq(q^{1/2}x), \text{~whence~} c_{k,n}(q) = q^{n/2} \gamma_{k,n},
$$
we see how both studies are easily related. Our main results consist in a series
of sufficient conditions on $k,n$ for the existence of a zero of $\gamma_{k,n}$
or, equivalently, of $c_{k,n}$, in the punctured open disk $\Do(0,1) \setminus \{0\}$.
However we confess that we could not solve the main mystery: \emph{which are the bad
values of $q$ ?}


\subsection{First computations}
\label{subsection:Premierscalculs}

Obviously $\gamma_{1,n}(q) = q^{n(n-1)/2}$ never vanishes. The same is true for
$\gamma_{2,n}(q)$ thanks to the following calculation:
\begin{align*}
\gamma_{2,n}(q) &= \sum_{l+m=n} q^{(l(l-1) + m(m-1))/2} = \sum_{l+m=n} q^{n(n-1)/2} q^{-lm} 
= q^{n(n-1)/2} \sum_m q^{m(m-n)} \\
&= q^{n(n-1)/2} \sum_m (q^2)^{m(m-1)/2} (q^{-n+1})^m = q^{n(n-1)/2} \, \theta_{q^2}(q^{-n+1}),
\end{align*}
which can only vanish for $0 < \lmod q \rmod < 1$ by virtue of the triple product
formula along with the impossibility of relation $-q^{-n+1} \in q^\Z$. \\

Computing $\gamma_{2,n}$ easily yields a closed formula for $\theta_q^2$:
\begin{equation}
\label{eqn:premiereidentite}
\theta_q^2(x) = \theta_{q^2}(q) \theta_{q^2}(x^2) + x \theta_{q^2}(1) \theta_{q^2}(q x^2).
\end{equation}
Here is a direct argument (more general arguments will come later, see sections
\ref{subsection:RNBIII} and \ref{subsection:Puissance4}):
\begin{align*}
\theta_q^2(x) &= \sum_{n \in \Z} \gamma_{2,n}(q) x^n = 
\sum_{n \in \Z} q^{n(n-1)/2} \, \theta_{q^2}(q^{-n+1}) x^n \\
&= \sum_{m \in \Z} q^{m(2m-1)} \, \theta_{q^2}(q^{-2m+1}) x^{2m} +
\sum_{m \in \Z} q^{m(2m+1)} \, \theta_{q^2}(q^{-2m}) x^{2m+1} \\
&= \sum_{m \in \Z} q^{m(2m-1)} \, q^{-m(m+1)} \, q^m \, \theta_{q^2}(q) x^{2m} +
   \sum_{m \in \Z} q^{m(2m+1)} \, q^{-m(m+1)} \, \theta_{q^2}(1) x^{2m+1} \\
&= \theta_{q^2}(q) \sum_{m \in \Z} q^{m(m-1)} x^{2m} +
   \theta_{q^2}(1) \sum_{m \in \Z} q^{m^2} x^{2m+1} 
= \theta_{q^2}(q) \theta_{q^2}(x^2) + x \theta_{q^2}(1) \theta_{q^2}(q x^2).
\end{align*}
We used \emph{en passant} the relation $\thq(q^{-m} x) = q^{-m(m+1)/2} x^m \thq(x)$
(with $q \leftarrow q^2$). \\

It is difficult not to hope that such an identity can be extended to higher powers. 
According to Professor Bruce Berndt, powers of theta functions are tackled in
volume III of the ``Ramanujan Notebooks'' \cite{BerndtRamanujanNBIII}; and indeed,
he there gives the proof (\emph{Entry 29}) of such a relation from which the above
equality flows easily; and \emph{Entry 30} seems to give the clue to the calculation
of $\theta_q^3(z)$ and maybe of $\theta_q^4(z)$, but it is not quite clear. At any
rate, using \emph{Entries} 18 to 30 does allow a first approach to our problem, see
section \ref{subsection:RNBIII}. \\

Another kind of calculus, akin to the theory of divergent $q$-series \cite{zh3} and
of their transformations ($q$-Borel, $q$-Laplace \dots), allows one to iteratively
obtain such identities. We expound it in section \ref{subsection:Puissance4}. With
a bit of real analysis one can then deduce for instance that $\gamma_{3,0}(q)$ vanishes
for at least one negative value of $q$. Using some numerical/formal calculus with Maple,
one can make the value more precise. Note that this section rather uses function
$\vartheta$ in order to simplify the form of the identities; translating back
the results is immediate. \\

It is possible to recognize in the series such as that giving $\gamma_{3,0}(q)$ the
generating series of the number of representations by some quadratic form:
$$
\gamma_{3,0}(q) = \sum_{a,b,c \in \Z \atop a+b+c = 0} q^{(a^2 - a + b^2 - b + c^2 - c)/2} =
\sum_{a,b,c \in \Z \atop a+b+c = 0} q^{-(ab + bc + ca)} = \sum_{a,b \in \Z} q^{a^2 + ab + b^2} = f(q),
$$
where we have set
$f(x) := \sum\limits_{a,b \in \Z} x^{a^2 + ab + b^2} = \sum\limits_{n \geq 0} r(n) x^n$ for
$\lmod x \rmod < 1$; here $r(n)$ denotes the number of pairs $(a,b) \in \Z \times \Z$
such that $a^2 + ab + b^2 = n$ (the computation is wholly justified by the fact that
this is a positive definite quadratic form). The integer:
$$
R(n) := r(0) + \cdots + r(n) =
\text{card} \{(a,b) \in \Z \times \Z \tq a^2 + ab + b^2 \leq n\}
$$
is equivalent, when $n \to + \infty$, to $C n$, where $C$ is the area of the ellipse
$x^2 + xy + y^2 \leq 1$, \ie\ $2 \pi/\sqrt{3}$. By standard abelian theorems we draw
that, when $x \to 1^-$:
$$
\dfrac{f(x)}{1-x} = \sum_{n \geq 0} R(n) x^n \sim \sum C n x^n = C \dfrac{x}{(1-x)^2}
\Longrightarrow f(x) \sim \dfrac{C}{1-x} \cdot
$$
On the other hand, the exponent $a^2 + ab + b^2$ is even if, and only if $a$ and $b$
are. The even part of series $f(x)$ is therefore the sub-sum on the such pairs, that
is $f(x^4)$, so that $f(-x) = 2 f(x^4) - f(x)$. Hence, still for $x \to 1^-$:
$$
f(x^4) \sim \dfrac{C}{1-x^4} \sim \dfrac{C}{4(1-x)} \Longrightarrow
f(-x) \sim \dfrac{-C}{2(1-x)} \Longrightarrow
\lim_{x \to -1 \atop x > -1} f(x) = - \infty.
$$
Since $f(0) = 1$, the function $f(x)$, thus also the function $\gamma_{3,0}(q)$, vanishes
at least once on $\left]-1,0\right[$. One of our goals is to generalize this fact and
the ideas which led to it.


\subsection{A bit of theory}
\label{subsection:unpeudetheorie}


\subsubsection{Linear relations among theta powers and theta functions}
\label{subsubsection:linearrelations}

We shall write various sets of integers as:
$$
\N = \{0,1,\ldots\}, \quad
\N^* = \N \setminus \{0\} = \{1,2,\ldots\}, \quad
\Z = \N \cup (-\N) = \{0,1,-1,2,-2,\ldots\}.
$$

Let $k \in \N^*$. The $\C$-linear space of holomorphic functions over $\Cs$ such that
$f(qx) = x^{-k} f(x)$ has dimension $k$, as can be seen by expanding $f$ into a Laurent
series $\sum f_m x^m$: the latter is uniquely and linearly determined by coefficients
$f_0,\ldots,f_{k-1}$. The functions $x^i \theta_{q^k}(q^i x^k)$, $i = 0,\ldots,k-1$, belong
to that space and are linearly independent (their supports are pairwise disjoint). The
function $\theta_q^k$ also belongs to that space, whence a relation:
\begin{equation}
\label{eqn:identitegenerale}
\theta_q^k(x) = \sum_{i=0}^{k-1} a_{k,i} \, x^i \, \theta_{q^k}(q^i x^k).
\end{equation}
For instance identity \eqref{eqn:premiereidentite} at the end of the calculation at the
beginning of \ref{subsection:Premierscalculs} says that $a_{2,0} = \theta_{q^2}(q)$ and
$a_{2,1} = \theta_{q^2}(1)$. \\

Expanding the left hand side along with each term of the right hand side yields the
relations:
\begin{equation}
\label{eqn:relationsentrelescoeffs}
n = j k + i, 0 \leq i \leq k-1 \Longrightarrow
\gamma_{k,n} = (q^k)^{j(j-1)/2} q^{ij} a_{k,i}.
\end{equation}
In particular $a_{k,i} = \gamma_{k,i}$ ($i = 0,\ldots,k-1$) so our identity above
boils down to:
$$
\theta_q^k(x) = \sum_{i=0}^{k-1} \gamma_{k,i} \, x^i \, \theta_{q^k}(q^i x^k).
$$
Note that the functional equation $\theta_q^k(qx) = x^{-k} \theta_q^k(x)$ immediately
yields $\gamma_{k,n+k} = q^n \gamma_{k,n}$ which implies, by iteration, formula
\eqref{eqn:relationsentrelescoeffs}. \\

Thus, on the one hand vanishing of some $\gamma_{k,n}(q)$ only depends of the class of
$n$ modulo $k$; on the other hand, the $k$ complex numbers $a_{k,i}$, $i = 0,\ldots,k-1$
contain all relevant information. We are therefore led to look for explicit identities
(\ie\ with known coefficients) of the above type.


\subsubsection{Some easy facts about the coefficients}
\label{subsubsection:easyfacts}

We collect here some facts about the coefficients $\gamma_{k,n}$ (and therefore also
about the coefficients $c_{k,n}$) which follow directly from the explicit formulas
\eqref {eqn:formuleexactecoeffstheta} and \eqref{eqn:formuleexactecoeffsvartheta}. \\

First note that $n(n-1) \geq 0$ for all $n \in \Z$, and that it vanishes only for
$n = 0,1$. So $\gamma_{k,n}(q)$ involves only non negative powers of $q$ and the
power $q^0$ comes from terms with multiindex $(m_1,\ldots,m_k) \in \{0,1\}^k$
such that $m_1 + \cdots + m_k = n$. There are ${k \choose n}$ such terms, whence:
\begin{equation}
\label{eqn:DASgamma(q)en0}
\forall k \in \N^* \;,\; \forall n \in \Z \;,\; \gamma_{k,n}(q) = {k \choose n} + O(q).
\end{equation}
Second, note that in the summation \eqref{eqn:formuleexactecoeffstheta}, changing
$(m_1,\ldots,m_k)$ to $(1-m_1,\ldots,1-m_k)$ does not change the general term
$q^{(\cdots)}$, permutes $\Z^k$ and transforms the sum $n$ into $k-n$. Therefore:
\begin{equation}
\label{eqn:symetriegamma}
\forall k \in \N^* \;,\; \forall n \in \Z \;,\; \gamma_{k,n} = \gamma_{k,k-n}.
\end{equation}
Last, we give without proof an identity somehow related to the above symmetry;
it can be checked by elementary manipulation of indexes:
$$
\forall k \in \N^* \;,\; \forall n \in \{1,\ldots,k-1\} \;,\;
\gamma_{k,n} = \sum_{m \in \Z} \gamma_{n,m} \gamma_{k-n,m}.
$$


\subsection{Contents of the paper}
\label{subsection:contents}

In order to have information on the coefficients, we first look for identities
in section \ref{section:identitiesandcoeffs}. \\
Using Ramanujan's formulas in \ref{subsection:RNBIII} gives some information on
powers of $\thq$, unhappily not enough. \\
So we develop in \ref{subsection:Puissance4} a calculus inspired from the theory
of transformations of divergent $q$-series \cite{zh3}. However, this is more simply
done using function $\vartheta$ instead of $\thq$ (but the results are quite easy
to translate). This calculus allows us to find some preliminary information on
the vanishing of coefficients with some real analysis. \\

In \ref{section:asymptoticsandvanishing} we generalize the real analysis argument
through an asymptotic study of the coefficients of $\thq^k$, resp. $\vartheta^k$,
near $q = -1$. \\
First, in \ref{subsection:CSAnnulation} we relate said coefficients to generating
series of the numbers of representations by quadratic forms. Usual considerations
of geometry of numbers yield rather good estimations (theorem \ref{thm:estimasymptfnk})
and a first list of sufficient conditions for vanishing coefficients (corollary
\ref{cor:CSAnnulation}). Last, well known modular properties of the generating series
allow for some complementary information (proposition \ref{prop:CScomplementaire}). \\
Last, in \ref{section:-1} we return to the formalism for $\vartheta$ introduced in
\ref{subsection:Puissance4} and use the classical modular transformation formula to
obtain again asymptotic estimates (theorems \ref{theo:gamma} and \ref{theo:limite})
from which we deduce more sufficient conditions: actually, exactly what is needed
to recover corollary \ref{cor:CSAnnulation} and its consequences. \\

The difference between the two approaches towards the same sufficient conditions is
reflected in the nature of the asymptotic estimates: one is in terms of powers of
$t = \log (e^{-\pi \ii} q)$, the other is a transseries in the exponential scale. \\

Note that in the end one mysterious question remains unsolved: what can be said
of the ``bad values'' of $q$ ? Contrary to our first guesses, they are not all
real negative \dots


\subsection*{Acknowledgements}

The first author is indebted to Professor Berndt for suggesting to look at Ramanujan's
Notebooks and positive comments on first attempts. \\

The second author was supported by Labex CEMPI (Centre Europ\'een pour les Math\'ematiques,
la Physique et leurs Interaction).


\subsection*{Some more notations}

We already introduced $\N,\N^*,\Z$ at the beginning of
\ref{subsubsection:linearrelations}. \\
For every $k \in \N^*$ and $a \in \Z$, we write $a \pmod{k}$ either for the remainder
of the euclidean division of $a$ by $k$ (thus $a \pmod{k} \in \{0,\ldots,k-1\}$) or
for the congruence class of $a$ modulo $k$ (thus $a \pmod{k} \in \Z/k\Z$). The context
should hopefully lift any ambiguity. \\
In the same spirit, if$f: \Z \rightarrow \C$ is any $k$-periodic function, one writes
$\sum\limits_{a \!\! \pmod{k}} f(a)$ a sum taken over an arbitrary set of
representatives $a$ modulo $k$. \\
Last, we shall write in \underline{underlined characters} $k$-uples
$\n = (n_1,\ldots,n_k)$, whether in $\Z^k$, $\R^k$ \dots



\section{Identities and coefficients}

\label{section:identitiesandcoeffs}


\subsection{Some identities drawn from ``Ramanujan Notebooks''}
\label{subsection:RNBIII}

Our sole reference in this section is chapter 1 of volume III of the Ramanujan Notebooks
\cite{BerndtRamanujanNBIII}. For non zero complex numbers $a,b$ such that $\lmod ab \rmod < 1$,
Ramanujan introduces the function:
$$
f(a,b) := 1 + \sum_{k = 1}^{+\infty} (ab)^{k(k-1)/2} (a^k + b^k) =
\sum_{k=-\infty}^{+\infty} a^{k(k+1)/2} b^{k(k-1)/2},
$$
in which can be recognized, in a very symmetric guise, our function $\thq$. Indeed, setting
$q := ab$, one immediately sees that:  
$$
f(a,b) = \thq(a).
$$
We shall translate in terms of $\thq$ some identities of Ramanujan (among those for which
this is not already done in \cite{BerndtRamanujanNBIII}) then draw consequences.


\subsubsection{Translation of some identities from \cite{BerndtRamanujanNBIII}}
\label{subsubsection:translationRNBIII}

We give no proof of these translations, verifications are mechanical (even if the
proof of the identities themselves are not !). Actually, \emph{Entry 22} is directly
copied from \cite{BerndtRamanujanNBIII}. \\

We omitted the important \emph{Entries} 20 and 27, which express modularity, because
we do not take advantage of that property in the present section (it will appear in
section \ref{section:asymptoticsandvanishing}).

\begin{ent}[18]
\begin{enumerate}
\item $\thq(x) = \thq(q/x)$
\item $\thq(q) = 2 \theta_{q^4}(q) = 2 \theta_{q^4}(q^3)$
\item $\thq(-1) = \thq(-q) = 0$
\item $\thq(x) = q^{n(n-1)/2} x^n \thq(q^n x)$
\end{enumerate}
\end{ent}

The above equalities are elementary; the following is Jacobi's celebrated Triple Product
Formula. Recall the Pochhammer symbols $(x;q)_\infty := \prod\limits_{n \geq 0} (1 - x q^n)$.

\begin{ent}[19]
$\thq(x) = (q;q)_\infty (-x;q)_\infty (-q/x;q)_\infty$
\end{ent}

\begin{ent}[22]
\begin{enumerate}
\item $\phi(q) := \theta_{q^2}(q) = \sum\limits_{k \in \Z} q^{k^2} =
\dfrac{(-q;q^2)_\infty (q^2;q^2)_\infty}{(q;q^2)_\infty (-q^2;q^2)_\infty}$
\item $\psi(q) := \theta_{q^4}(q) = \sum\limits_{k =0}^\infty q^{k(k+1)/2} =
\dfrac{(q^2;q^2)_\infty}{(q;q^2)_\infty}$
\item $f(-q) := \theta_{q^3}(-q) = \sum\limits_{k \in \Z} (-1)^k q^{k(3k-1)/2} = (q;q)_\infty$
\item $\chi(q) := (-q;q^2)_\infty$ (note that this item is purely a definition)
\end{enumerate}
\end{ent}

The following one is related to the \emph{distribution formula} of Pochhammer symbols:
$\prod\limits_{k=0}^{n-1} (q^k x;q^n)_\infty = (x;q)_\infty$.

\begin{ent}[28]
$\prod\limits_{k=0}^{n-1} \theta_{q^n}(q^k x) = \dfrac{(q^n;q^n)_\infty^n}{(q;q)_\infty} \thq(x)$
\end{ent}

\emph{Entry 29} has been the most fuitful for our goals. The apparent dissymmetry
in $x$,$y$ disappears if one takes in account \emph{Entry 18}.

\begin{ent}[29]
\begin{enumerate}
\item $\thq(x) \thq(y)  + \thq(-x) \thq(-y) = 2 \theta_{q^2}(xy) \theta_{q^2}(qy/x)$ 
\item $\thq(x) \thq(y)  - \thq(-x) \thq(-y) = 2 x \theta_{q^2}(qxy) \theta_{q^2}(y/x)$ 
\end{enumerate}
\end{ent}

\begin{ent}[30]
\begin{enumerate}
\item $\theta_{q^2}(x) \theta_{q^2}(qx) = \thq(x) \psi(q)$
\item $\thq(x) + \thq(-x) = 2 \theta_{q^4}(qx^2)$
\item $\thq(x) - \thq(-x) = 2 x \theta_{q^4}(q^3 x^2)$
\item $\thq(x) \thq(-x) = \theta_{q^2}(-x^2) \phi(-q)$
\item $\theta_q^2(x) + \theta_q^2(-x) = 2 \theta_{q^2}(x^2) \phi(q)$
\item $\theta_q^2(x) - \theta_q^2(-x) = 4 x \theta_{q^2}(q x^2) \psi(q^2)$
\end{enumerate}
\end{ent}

The following corollary of \emph{Entry 30} (actually, seing its proof, rather of
\emph{Entry 29}) looks promising:
$$
\thq(x) \thq(y) \thq(nx) \thq(ny)  + \thq(-x) \thq(-y) \thq(-nx) \thq(-ny) =
2 x \thq(y/x) \thq(nxy) \thq(n) \psi(q).
$$
However, to exploit it, one would have wanted a formula involving the \emph{difference}
$$
\thq(x) \thq(y) \thq(nx) \thq(ny) - \thq(-x) \thq(-y) \thq(-nx) \thq(-ny)
$$
and we found no simple such formula. \\

At any rate, it seems difficult to find in \cite{BerndtRamanujanNBIII} equalities
involving products of functions $\thq$ with different bases $q$, which explains in
part the limited range of the following consequences with respect to our goals.


\subsubsection{Consequences for the powers of $\thq$}

One first proves anew identity \eqref{eqn:premiereidentite}:

\begin{prop}
$$
\theta_q^2(x) = \theta_{q^2}(q) \theta_{q^2}(x^2) + x \theta_{q^2}(1) \theta_{q^2}(qx^2).
$$
\end{prop}
\begin{proof}
The most direct path is to take the arithmetical mean of both items of \emph{Entry 29},
with $y = x$. One can also take the arithmetical mean of the last two items of
\emph{Entry 30}, then replace $\phi(q)$ and $\psi(q^2)$ according to the two first items
of \emph{Entry 22}, last apply the first two items of \emph{Entry 18}.
\end{proof}

As already noted, $a_{2,0} = \theta_{q^2}(q)$ et $a_{2,1} = \theta_{q^2}(1)$
(notations of \ref{subsection:unpeudetheorie}).

\begin{cor}
\label{cor:deuxiemeidentite}
\begin{align*}
a_{4,0} &= \theta_{q^2}(q)^2 \theta_{q^4}(q^2) + q \theta_{q^2}(1)^2 \theta_{q^4}(1), \\
a_{4,1} &= 2 \theta_{q^2}(1) \theta_{q^2}(q) \theta_{q^4}(q^3), \\
a_{4,2} &= \theta_{q^2}(q)^2 \theta_{q^4}(1) + \theta_{q^2}(1)^2 \theta_{q^4}(q^2), \\
a_{4,3} &= 2 \theta_{q^2}(1) \theta_{q^2}(q) \theta_{q^4}(q).
\end{align*}
\end{cor}
\begin{proof}
We square \eqref{eqn:premiereidentite}:
$$
\theta_q^4(x) =
(\theta_{q^2}(q))^2 (\theta_{q^2}(x^2))^2 + x^2 (\theta_{q^2}(1))^2 (\theta_{q^2}(qx^2))^2
+ 2 x \theta_{q^2}(1) \theta_{q^2}(q) \theta_{q^2}(x^2) \theta_{q^2}(qx^2).
$$
First term, directly from \eqref{eqn:premiereidentite} (replacing $q$ by $q^2$ and $x$ by $x^2$):
$$
(\theta_{q^2}(x^2))^2 = \theta_{q^4}(q^2) \theta_{q^4}(x^4) + x^2 \theta_{q^4}(1) \theta_{q^4}(q^2 x^4)
$$
Second term, directly from \eqref{eqn:premiereidentite} (replacing $q$ by $q^2$ and $x$ by $q x^2$):
$$
(\theta_{q^2}(qx^2))^2 =
\theta_{q^4}(q^2) \theta_{q^4}(q^2 x^4) + q x^2 \theta_{q^4}(1) \theta_{q^4}(q^4 x^4)=
\theta_{q^4}(q^2) \theta_{q^4}(q^2 x^4) + q x^{-2} \theta_{q^4}(1) \theta_{q^4}(x^4)
$$
Third term, appealling to \emph{Entry 29} with $x \leftarrow x^2$ and
$y \leftarrow q x^2$ (mean of both items):
$$
\theta_{q^2}(x^2) \theta_{q^2}(qx^2) =
\theta_{q^4}(q^3) \theta_{q^4}(q x^4) + x^2 \theta_{q^4}(q) \theta_{q^4}(q^3 x^4).
$$
Collecting:
\begin{align*}
\theta_q^4(x) &= (\theta_{q^2}(q))^2
\left(\theta_{q^4}(q^2) \theta_{q^4}(x^4) + x^2 \theta_{q^4}(1) \theta_{q^4}(q^2 x^4)\right) \\
&+ x^2 (\theta_{q^2}(1))^2
\left(\theta_{q^4}(q^2) \theta_{q^4}(q^2 x^4) + q x^{-2} \theta_{q^4}(1) \theta_{q^4}(x^4)\right) \\
&+ 2 x \theta_{q^2}(1) \theta_{q^2}(q)
\left(\theta_{q^4}(q^3) \theta_{q^4}(q x^4) + x^2 \theta_{q^4}(q) \theta_{q^4}(q^3 x^4)\right) \\
&=
\left(\theta_{q^2}(q)^2 \theta_{q^4}(q^2) + q \theta_{q^2}(1)^2 \theta_{q^4}(1)\right)
\theta_{q^4}(x^4) \\
&+ 2 \theta_{q^2}(1) \theta_{q^2}(q) \theta_{q^4}(q^3) x \theta_{q^4}(q x^4) \\
&+ \left(\theta_{q^2}(q)^2 \theta_{q^4}(1) + \theta_{q^2}(1)^2 \theta_{q^4}(q^2)\right)
x^2 \theta_{q^4}(q^2 x^4) \\
&+ 2 \theta_{q^2}(1) \theta_{q^2}(q) \theta_{q^4}(q) x^3 \theta_{q^4}(q^3 x^4).
\end{align*}
\end{proof}

Note that, since $\theta_{q^4}(q^3) = \theta_{q^4}(q)$, we have $a_{4,1} = a_{4,3}$ as was
to be expected by \eqref{eqn:symetriegamma}. Also note that various transformations are
possible; for instance, using \emph{Entry 18} one can check that
$$
a_{4,1} = a_{4,3} = \dfrac{1}{2} \thq(1)^3.
$$
In all cases, these coefficients vanish for no value of $q$. However, corollary
\ref{cor:CSAnnulation} in \ref{subsection:CSAnnulation} predicts that $a_{4,0}$ and
$a_{4,3}$ vanish for at least one $q < 0$ (but says nothing about $a_{4,1}$ and $a_{4,3}$). \\

The process can be (partially) generalized by squaring identity
\eqref{eqn:identitegenerale}.
By definition:
\begin{align*}
\theta_q^{2k}(x) &= \left(\sum_{i=0}^{k-1} a_{k,i} x^i \theta_{q^k}(q^i x^k)\right)^2 \\
&= \sum_{i,j=0}^{k-1} a_{k,i} a_{k,j} x^{i+j} \theta_{q^k}(q^i x^k) \theta_{q^k}(q^j x^k) \\
&= \sum_{i,j=0}^{k-1} a_{k,i} a_{k,j} x^{i+j} 
\left(\theta_{q^{2k}}(q^{i+j} x^{2k}) \theta_{q^{2k}}(q^{j-i+k}) + 
q^i x^k \theta_{q^{2k}}(q^{i+j+k} x^{2k}) \theta_{q^{2k}}(q^{j-i})\right).
\end{align*}
We used again consequence 
$\thq(x) \thq(y) = \theta_{q^2}(xy) \theta_{q^2}(qy/x) + x \theta_{q^2}(qxy) \theta_{q^2}(y/x)$ 
of \emph{Entry 29}. Replacing each term $\theta_{q^{2k}}(q^m y)$ such that $m \geq 2k$ by
$y^{-1} \theta_{q^{2k}}(q^{m-2k} y)$ and collecting carefully, one finds, for $l = 0,\ldots,2k-1$:
$$
a_{2k,l} = \sum_{i+j=l} a_{k,i} a_{k,j} \theta_{q^{2k}}(q^{k+j-i}) +
\begin{cases}
\sum_{i+j+k=l} a_{k,i} a_{k,j} q^{k-j} \theta_{q^{2k}}(q^{j-i}) \text{~if~} l \leq k-2, \\
\sum_{i+j+k=l} a_{k,i} a_{k,j} q^i \theta_{q^{2k}}(q^{j-i}) \text{~if~} l \geq k.
\end{cases}
$$
(There is no second term $\sum$ if $l = k-1$.) Taking $k = 2$ and using the values of 
$a_{2,0}$, $a_{2,1}$ given in \ref{subsection:unpeudetheorie}, one does recognize the
$a_{4,i}$ of corollary \ref{cor:deuxiemeidentite}. Iterative application of these formulas
would allow us to determine the $a_{k,i}$ when $k$ is a power of $2$; but even so, without
giving a closed formula. The methods of \ref{subsection:Puissance4} will prove more
to be effective.


\subsection{A calculus of coefficients of powers of theta}
\label{subsection:Puissance4}


\subsubsection{Notations and conventions for this subsection}
\label{subsubsection:conventionsPuissance4}

We suppose chosen $\tau \in \C$ such that $q=e^{2\pi i\tau}$ and $\Im\tau>0$, which
in particular gives an inambiguous meaning to non integral powers of $q$. In this
subsection, in order to have simpler formulas, we rather consider the following
theta function:
$$
\vartheta(x) := \vartheta_q(x) := \sum_{n\in\Z} q^{n^2/2} x^n = \thq\left(q^{1/2} x\right)\,.
$$
The series $\vartheta$ plainly represents an analytic function all over the punctured
complex plane. Let $k\in\N^*$ and $n\in\Z$. We write $c_{k,n}(q)$ the coefficient
of $x^n$ in the Laurent series expansion of $\vartheta^k(x)$; that is:
\begin{equation}
\label{equation:ckl}
c_{k,n}(q) = \sum_{(n_1,\ldots,n_k) \in \Z^k \atop n_1 + \cdots + n_k = n} q^{(n_1^2+\cdots+n_k^2)/2}.
\end{equation}
We therefore have:
\begin{equation}
\label{eqn:lienentrediverscoeffs}
c_{k,n} = q^{n/2} \, \gamma_{k,n}.
\end{equation}

Setting $n := 0$ in formula \eqref{equation:ckl}, we get:
\begin{equation}
\label{equation:ck20}
c_{k,0}(q) = 1 + k(k-1) q + O(q^2),
\end{equation}
where the constant term corresponds to $(0,\ldots,0)$ and where the coefficient $k(k-1)$
before $q$ comes from the fact that $k$-uples $(n_1,\ldots,n_k)$ such that
$n_1^2 + \cdots + n_k^2 = 2$, $n_1 + \cdots + n_k = 0$ are made of $(k-2)$ zeroes, one $+1$ and
one $-1$. More generally, after \eqref{eqn:DASgamma(q)en0} page \pageref{eqn:DASgamma(q)en0}:
\begin{equation}
\label{equation:ck21}
c_{k,n}(q) = q^{n/2} \,\left(\binom k n + O(q)\right).
\end{equation}

Actually, $\gamma_{k,n}(q) = q^{-n/2} \, c_{k,n}(q)$ represents an analytic function over
the unit disk, with value $\displaystyle \binom kn$ at $0$ and its Taylor series at $0$
has non negative integer coefficients.

\paragraph{Some notations.}

\begin{enumerate}
\item If $f = \displaystyle\sum_{n\in\Z} a_n x^n \in \C\left[\left[x,\frac1x\right]\right]$,
one puts:
$$
\cB f(x) := \cB_q f(x) =
\displaystyle \sum_{n\in\Z} a_n q^{n^2/2} x^n \in \C\left[\left[x,\frac1x\right]\right].
$$
\item If $\cB_q f$ represents a holomorphic function over $\C^*$, one puts:
$$
\cT f(x) = \cT _q f(x) =
\displaystyle \sum_{n\in\Z} \cB_qf(q^{-n}) x^n \in \C\left[\left[x,\frac1x\right]\right].
$$ 
\item For all $a \in \C^*$ and $k \in \N^*$, one denotes $\sigma_a f(x) := f(ax)$ and
$\rho_k f(x) := f(x^k)$.
\item  For $(a_1, ..., a_m) \in \C^m$, one puts:
$$
(a_1,...,a_m;q)_\infty := \prod_{n=1}^m (a_n; q)_\infty.
$$
(Pochhammer symbols $(a;q)_\infty$ were defined at the beginning of
\ref{subsubsection:translationRNBIII}.)
\end{enumerate}

\begin{rem}
\label{rq1}
\rm We here recall some useful facts:
\begin{enumerate}
\item
\label{thetaequation}
$\displaystyle\vartheta(1/x) = \vartheta(x)$;
 $\displaystyle\vartheta(q^nx) = q^{-n^2/2}x^{-n}\vartheta(x)$ for $n\in\Z$.
\item
\label{3produit}
$\displaystyle\vartheta(x) = (q,-\sqrt q\,x,-\frac{\sqrt q}{x};q)_\infty$.
\item
\label{thetazero}
$\vartheta(x)=0$ if, and only if $x\displaystyle\in -q^{\Z+\frac12}$.
\item
\label{split}
$\displaystyle\prod_{n=0}^{k-1}\vartheta_{q^k}(xq^{n}) =
\frac{(q^k;q^k)_\infty^k}{(q;q)_\infty}\,\vartheta(xq^{(k-1)/2})$
for $k\in\N^*$.
\item
\label{thetaB}
$\cB_{q^k}\vartheta_{q^m}=\vartheta_{q^{k+m}}$ for $k$, $m\in\N^*$.
\item
\label{x}
$\displaystyle\cB(x^n f)=q^{n^2/2}x^n\cB f$ for $n\in\Z$.
\item
$\sigma_a\cB=\cB\sigma_a$ for $a\in\C^*$.
\item
$\cB_q\rho_k=\rho_k\cB_{q^{k^2}}$ for $k\in\N^*$.
\end{enumerate}
\end{rem}

\begin{prop}
\label{prop1}
\rm
If $\cB_q f$ represents a holomorphic function over $\C^*$, one has: $\vartheta f = \cB(\cT f)$.
\end{prop}
\begin{proof}
Immediate calculation.
\end{proof}

For $f=\vartheta$, remark \ref{rq1}~\eqref{thetaB} says that $\cB\vartheta = \vartheta_{q^2}$,
which implies:
$$
\cT\vartheta(x)=\sum_{n\in\Z}\vartheta_{q^2}(q^{-n})x^n=\sum_{n\in\Z}\vartheta_{q^2}(q^{n})x^n\,.
$$
With the help of proposition \ref{prop1}, one finds the following expression:
\begin{equation}
\label{equation:2}
\vartheta^2(x) = \sum_{n\in\Z} \vartheta_{q^2}(q^{n})q^{n^2/2}\,x^n.
\end{equation}
According to remark \ref{rq1}~\eqref{thetazero}, we know that $\vartheta_{q^2}(q^n) \not= 0$
except if $q^n\in-q^{2\Z+1}$, which is impossible since$-1\notin q^\Z$. Therefore:

\begin{fact}
The series expanding $\vartheta^2$ has no null coefficient.
\end{fact}

In order to go on, let us introduce some more general notations. Let $k\in\N^*$.
Replacing $n$ by $1$ and then $\vartheta$ by $\vartheta^k$ in the second functional
equation of remark \ref{rq1}~\eqref{thetaequation}, we find that
$\sigma_q\vartheta^k(x)=q^{-k/2}x^{-k}\vartheta^k(x)$. Moreover, if
$T_{k,n}(x)=x^n\vartheta_{q^k}(x^kq^n)$ with $n\in\Z$, one has
$\sigma_qT_{k,n}(x)=q^{-k/2}x^{-k}T_{k,n}(x)$. This being said, the functions $T_{k,n}$ all
satisfy the same linear equation as $\vartheta^k$ (compare \ref{subsubsection:linearrelations}).


\subsubsection{A recurrence relation}

\begin{prop}\label{prop2}
For every $k\in\N^*$, we have:
\begin{equation}\label{equation:k}
 \vartheta^k(x)=\sum_{n=0}^{k-1}c_{k,n}(q)\,x^n\,\vartheta_{q^k}(x^kq^n)\,.
\end{equation}
Moreover, the coefficients $c_{k,n}(q)$ satisfy the following recurrence relation:
\begin{equation}\label{equation:ck}
  c_{k+1,n}(q) =
  \sum_{n'=0}^{k-1}c_{k,n'}(q)\,q^{(n-n')^2/2}\,\vartheta_{q^{k^2+k}}(q^{-kn+(k+1)n'})\,.
\end{equation}
\end{prop}
\begin{proof}
We shall apply proposition \ref{prop1} to a function of the form
$f(x)=\vartheta_{q^k}(x^kq^n)$, where $k\in\N^*$ et $n\in\{0,...,k-1\}$.
Write $\vartheta_{q^k}(x^kq^n)=\sigma_{q^{n/k}}\rho_k\vartheta_{q^k}(x)$ and use remark
\ref{rq1}. We have:
$$
\cB\vartheta_{q^k}(x^kq^n)=\cB\sigma_{q^{n/k}}\rho_k\vartheta_{q^k}(x)= \cdots
=\vartheta_{q^{k^2+k}}(x^kq^n)\,.$$
We deduce that:
\begin{equation}
 \label{equation:kl}
 \vartheta(x)\,\vartheta_{q^k}(x^kq^n)=\cB(\cT f)(x)=
 \sum_{m\in\Z}\vartheta_{q^{k^2+k}}(q^{-mk+n})\,q^{m^2/2}\,x^n\,.
\end{equation}
If $m=(k+1)\nu+r$ with $\nu\in\Z$ and $r\in\{-n,1-n,...,k-n\}$, we write
$q^{-mk+n}=q^{-(k+1)k\nu-kr+n}$ and get:
$$
\displaystyle \vartheta_{q^{k^2+k}}(q^{-mk+n}) =
q^{-k(k+1)\nu^2/2}\,q^{(-kr+n)\nu}\,\vartheta_{q^{k^2+k}}(q^{-kr+n})\,.
$$
With $m^2=(k+1)^2\nu^2+2(k+1)r\nu+r^2$, it follows that:
$$
\vartheta_{q^{k^2+k}}(q^{-mk+n})\,q^{m^2/2}=
q^{(k+1)\nu^2/2}\,q^{(r+n)\nu}\,q^{r^2/2}\,\vartheta_{q^{k^2+k}}(q^{-kr+n})\,.
$$
Hence expression \eqref{equation:kl} becomes:
\begin{equation}
 \label{equation:kl0}
 \vartheta(x)\,\vartheta_{q^k}(x^kq^n) =
 \sum_{r=-n}^{k-n}q^{r^2/2}\,\vartheta_{q^{k^2+k}}(q^{-kr+n})\,x^r\,\vartheta_{q^{k+1}}(xq^{r+n})\,.
\end{equation}
Using \eqref{equation:k}, we deduce that \eqref{equation:kl0}:
$$
\vartheta^{k+1}(x) = \sum_{r=-n}^{k-n}q^{r^2/2}
\sum_{n=0}^{k-1} c_{k,n}(q)\,\vartheta_{q^{k^2+k}}(q^{-kr+n})\,x^{r+n}\,\vartheta_{q^{k+1}}(xq^{r+n}).
$$
This completes the proof of the recurrence relation \eqref{equation:ck}.
In this way, we finish the proof of proposition \ref{prop2}.
\end{proof}


\subsubsection{A symmetry relation}
\label{section:sym}

As already noted in \ref{subsubsection:easyfacts}, equation \eqref{eqn:symetriegamma},
$\gamma_{k,n}=\gamma_{k,k-n}$. Therefore, iff $0<n<k$, one has the following symmetry:
\begin{equation}
  \label{equation:Sym}
  c_{k,n}(q)=q^{n-k/2}\,c_{k,k-n}(q)\,.
\end{equation}

Let us explain how to use \eqref{equation:Sym} so as to shorten the expression 
of recurrence \eqref{equation:ck}. First assume that $k$ is odd $\ge 1$, and write $k=2k'+1$.
Grouping the terms in the sum from \eqref{equation:ck} according to whether $n'=0$ or
$1\le n'\le k'$ or $n'> k'$, one has:
\begin{equation}
\label{equation:ckk'}
\begin{split}
c_{k+1,n}(q) =
c_{k,0}(q)\,q^{n^2/2}&\,\vartheta_{q^{k^2+k}}(q^{-kn}) +
\sum_{n'=1}^{k'}\Bigl(c_{k,n'}(q)\,q^{(n-n')^2/2}\,\vartheta_{q^{k^2+k}}(q^{-kn+(k+1)n'}) \\
&+c_{k,k'+n'}(q)\,q^{(n-k'-n')^2/2}\,\vartheta_{q^{k^2+k}}(q^{-kn+(k+1)(k'+n')})\Bigr)\,.
\end{split}
\end{equation}
From \eqref{equation:Sym}, one has: $c_{k,k'+n'}(q)=q^{n'-1/2}\,c_{k,k'+1-n'}$.
Writing $k'+n'=k-1-k'+n'$, we draw $(k+1)(k'+n')=(k+1)k+(k+1)(n'-k'-1)$,
which yields:
\begin{equation*}
\begin{split}
\vartheta_{q^{k^2+k}}(q^{-kn+(k+1)(k'+n')})&=\vartheta_{q^{k^2+k}}(q^{-kn+(k+1)(n'-k'-1)+(k+1)k})  \\
&=q^{kn-(k+1)(n'-1/2)}\,\vartheta_{q^{k^2+k}}(q^{-kn+(k+1)(n'-k'-1)})\,.
\end{split}
\end{equation*}
Since $kn-(k+1)(n'-1/2)=(2k'+1)(n-n')+k'+1-n'$, one has:
$$q^{n'-1/2}\,
q^{(n-k'-n')^2/2}\,q^{kn-(k+1)(n'-1/2)}=q^{(n-n'+k'+1)^2/2}\,.
$$
As a consequence, if $n''=k'+1-n'$, one finds:
\begin{equation*}
 \begin{split}
c_{k,k'+n'}(q)\,q^{(n-k'-n')^2/2}&\,\vartheta_{q^{k^2+k}}(q^{-kn+(k+1)(k'+n')})\cr
&=c_{k,n''}\,q^{(n+n'')^2/2}\,\vartheta_{q^{k^2+k}}(q^{-kn-(k+1)n''})\,.
 \end{split}
\end{equation*}
When $n'$ goes from $1$ to $k'$, $n''$ runs along the same values in descending order.
Going back to formula \eqref{equation:ckk'}, one is led to the following remark.

\begin{fact}
If $k=2k'+1$ and $k' \in \N$, then:
\begin{equation}
\label{equation:ck_impair}
\begin{split}
c_{k+1,n}(q)=&c_{k,0}(q)\,q^{n^2/2}\,\vartheta_{q^{k^2+k}}(q^{kn})+\sum_{n'=1}^{k'}c_{k,n'}(q)\,\times\\
&\Bigl(q^{(n'-n)^2/2}\,\vartheta_{q^{k^2+k}}(q^{k(n'-n)+n'})+q^{(n'+n)^2/2}\,\vartheta_{q^{k^2+k}}(q^{k(n'+n)+n'})\Bigr)\,.
\end{split}
\end{equation}
In particular, for $n=0$, one has:
\begin{equation}\label{equation:ck_impair0}
c_{k+1,0}(q)=c_{k,0}(q)\,\vartheta_{q^{k^2+k}}(1)+2\sum_{n=1}^{k'}c_{k,n}(q)\,\,q^{n^2/2}\,\vartheta_{q^{k^2+k}}(q^{(k+1)n})\,.
\end{equation}
\end{fact}

Consider now the case when $k=2k'$ with $k' \in \N^*$. Using relation \eqref{equation:ck},
one can check that 
\begin{equation}
  \label{equation:ckk'1}
\begin{split}
c_{k+1,n}(q)=c_{k,0}(q)\,&q^{n^2/2}\,\vartheta_{q^{k^2+k}}(q^{-kn})+c_{k,k'}(q)\,q^{(n-k')^2/2}\,\vartheta_{q^{k^2+k}}(q^{-kn+(k+1)k'})\cr
+\sum_{n'=1}^{k'-1}&\Bigl(c_{k,n'}(q)\,q^{(n-n')^2/2}\,\vartheta_{q^{k^2+k}}(q^{-kn+(k+1)n'})\\
&+c_{k,k'+n'}(q)\,q^{(n-k'-n')^2/2}\,\vartheta_{q^{k^2+k}}(q^{-kn+(k+1)(k'+n')})\Bigr)\,.
\end{split}
\end{equation}
A similar analysis would entail the following remark.

\begin{fact}
 \label{remarque:ck_pair} If $k = 2k'$ et $k' \in \N^*$, then:
 \begin{equation}\label{equation:ck_pair}
\begin{split}
c_{k+1,n}(q)=&c_{k,0}(q)\,q^{n^2/2}\,\vartheta_{q^{k^2+k}}(q^{kn})+c_{k,k'}(q)\,q^{(n-k')^2/2}\,\vartheta_{q^{k^2+k}}(q^{k(k'-n)+k'})\\
&+\sum_{n'=1}^{k'-1}c_{k,n'}(q)\,\Bigl(q^{(n'-n)^2/2}\,\vartheta_{q^{k^2+k}}(q^{k(n'-n)+n'})\\&\qquad\qquad\qquad+q^{(n'+n)^2/2}\,\vartheta_{q^{k^2+k}}(q^{k(n'+n)+n'})\Bigr)\,.
\end{split}
\end{equation}
In particular, for $n=0$, one has:
\begin{equation}\label{equation:ck_pair0}
\begin{split}
c_{k+1,0}(q)=&c_{k,0}(q)\,\vartheta_{q^{k^2+k}}(1)+c_{k,k'}(q)\,q^{k^2/8}\,\vartheta_{q^{k^2+k}}(q^{(k+1)k/2})\\&+2\sum_{n=1}^{k'-1}c_{k,n}(q)\,\,q^{n^2/2}\,\vartheta_{q^{k^2+k}}(q^{(k+1)n})\,. 
\end{split}
\end{equation}
\end{fact}


\subsubsection{About coefficients $c_{k,n}(q)$ for $k=2$ or $3$}

Putting $k := 1$ in \eqref{equation:ck}, since $c_{1,0}(q)=1$, we find:
\begin{equation}
 \label{equation:c2}
c_{2,n}(q)=\vartheta_{q^2}(q^n)\,q^{n^2/2}\,,\quad n\in\{0,1\}\,. 
\end{equation}
Moreover, with $k=2$ and $\mu=-1$ in \eqref{equation:ck1}, we have:
$$
c_{2,n}(q)=
\frac{q^{n^2/2}\,x^{-n}}{2\,\vartheta_{q^2}(x^2)}\,\left(\vartheta^2(xq^{-n/2})+(-1)^n\,\vartheta^2(-xq^{-n/2})\right)\,.
$$
Let us set $x=q^{(n+1)/2}$, and recall that $\vartheta(-\sqrt q)=0$; one deduces the expression
$\displaystyle
c_{2,n}(q)=\frac{q^{-n/2}\,\vartheta^2(\sqrt q)}{2\,\vartheta_{q^2}(q^{n+1})}\,.
$
Comparing the latter with \eqref{equation:c2} for $n=0$ gives the following identity:
\begin{equation}
 \label{equation:c22}
\vartheta^2(\sqrt q)=2\,\vartheta_{q^2}(1)\,\vartheta_{q^2}(q)\,.
\end{equation}
Applying relation \ref{rq1}~\eqref{split} for $k=2$ and $x=1$, one has:
$$
\vartheta_{q^2}(1)\,\vartheta_{q^2}(q)=
\frac{(q^2;q^2)_\infty^2}{(q;q)_\infty}\,\vartheta(\sqrt q)
=\frac{(q^2;q^2)_\infty^2}{(q;q)_\infty}\,\vartheta(\sqrt q)\,.
$$
Identity \eqref{equation:c22} can then be read as follows:
\begin{equation}
 \label{equation:c23}
\vartheta(\sqrt q)=2\,\frac{(q^2;q^2)_\infty^2}{(q;q)_\infty}\,,
\end{equation}
where the factor $2$ looks almost unexpected \dots \\

Let us compute $c_{3,n}(q)$ for $n=0,1,2$ with the help of \eqref{equation:ck}:
$$
c_{3,n}(q) =
c_{2,0}(q)\,q^{n^2/2}\,\vartheta_{q^{6}}(q^{2n})+c_{2,1}(q)\,q^{(n-1)^2/2}\,\vartheta_{q^{6}}(q^{2n-3})\,.
$$
With \eqref{equation:c2}, one finds:
\begin{equation}
 \label{equation:c3}
c_{3,n}(q)=q^{n^2/2}\,\vartheta_{q^2}(1)\,\vartheta_{q^{6}}(q^{2n})+q^{\left(1+(n-1)^2\right)/2}\,\vartheta_{q^2}(q)\,\vartheta_{q^{6}}(q^{2n-3})\,.
\end{equation}
In particular, one notes that:
$$
c_{3,0}(q)=\vartheta_{q^2}(1)\,\vartheta_{q^{6}}(1)+q\,\vartheta_{q^2}(q)\,\vartheta_{q^{6}}(q^{3})\,.
$$
If one puts $\displaystyle f(q)=\vartheta_{q^2}(1)$ and $\displaystyle g(q)=\vartheta_{q^2}(q)$,
one has:
\begin{equation}
 \label{equation:c30}
c_{3,0}(q)=f(q)\,f(q^3)+q\,g(q)\,g(q^3)\,.
\end{equation}

The functions $f$ and $g$ may be written in the following form:
$$
f(q)=
1+2\displaystyle\sum_{n\ge 1}q^{n^2}\,,\quad g(q)=2+2\displaystyle\sum_{n\ge 1}q^{n(n+1)}\,,
$$
where the power series all have the unit circle as convergence boundary. If $q\to-1^-$,
one sees that $|f(q)|<1$ but $g(q)\to+\infty$. Considering expression \eqref{equation:c30},
one draws that $c_{3,0}(q)\to-\infty$ for $q\to-1^-$. Since $c_{3,0}(0)=1$, one infers:

\begin{fact}
$c_{3,0}(q)$  admits (at least) a zero over $]-1,0[$.
\end{fact}

Using Maple, by chosing values of $n$ up to $N=600$, one more or less gets the graph of
$q\mapsto c_{3,0}(q)$ over interval $]-0.5,0.5[$. Depending on the look of the graph,
one can ask Maple to evaluate $c_{3,0}(q)$ for  $a=-0.163034$ and $b=-0.163033$.
Herebelow some of the numerical values. \\

For $N=400$:
$$
c_{3,0}(a)=-2.96590*10^{-6}\,,\quad c_{3,0}(b)=3.41022854*10^{-6}\,.
$$

For $N=500$:
$$
c_{3,0}(a)=-2.96589725*10^{-6}\,,\quad c_{3,0}(b)=3.41023*10^{-6}\,.
$$

For $N=600$:
$$
c_{3,0}(a)=-2.96589725*10^{-6}\,,\quad c_{3,0}(b)=3.41022854*10^{-6}\,.
$$

\begin{rem}
It seems that $c_{3,0}(q)$ vanishes somewhere between $a=-0.163034$ and $b=-0.163033$.
\end{rem}

Put $n=1$ in \eqref{equation:c3}; then:
\begin{equation}
\label{equation:c31}
c_{3,1}(q)=
q^{1/2}\,\left(\vartheta_{q^2}(1)\,\vartheta_{q^{6}}(q^{2})+\vartheta_{q^2}(q)\,\vartheta_{q^{6}}(q)\right)\,.
\end{equation}
We find likewise:
$$
c_{3,2}(q)=
q\left(q\,\vartheta_{q^2}(1)\,\vartheta_{q^{6}}(q^{4})+\vartheta_{q^2}(q)\,\vartheta_{q^{6}}(q)\right)\,.
$$
Noting that $\vartheta_{q^{6}}(q^{4})=\vartheta_{q^{6}}(q^{-4})=q^{-1}\,\vartheta_{q^{6}}(q^{2})$,
we get:
\begin{equation}
\label{equation:c32}c_{3,2}(q)=q^{1/2}\,c_{3,1}(q)\,.
\end{equation}
This is consistent with the symmetry stated in \eqref{equation:Sym}.


\subsubsection{About coefficients $c_{4,n}(q)$}

Putting $k'=1$ and $k=3$ in \eqref{equation:ck_impair}, we obtain that:
\begin{equation}
 \label{equation:cpas40}
\begin{split}
c_{4,n}(q)=&c_{3,0}(q)\,q^{n^2/2}\,\vartheta_{q^{12}}(q^{kn})+c_{3,1}(q)\,\times\\
&\Bigl(q^{(n-1)^2/2}\,\vartheta_{q^{12}}(q^{-3n+4})+q^{(n+1)^2/2}\,\vartheta_{q^{12}}(q^{3n+4})\Bigr)\,.
\end{split}
\end{equation}
For $n=0$, formula \eqref{equation:ck_impair0} implies:
$$
c_{4,0}(q)=c_{3,0}(q)\,\vartheta_{q^{12}}(1)+2\,c_{3,1}(q)\,\,q^{1/2}\,\vartheta_{q^{12}}(q^{4})\,.
$$
In account of \eqref{equation:c3}, one can check that:
\begin{equation}
 \label{equation:c40}
 \begin{split}\displaystyle
 c_{4,0}(q)=&\vartheta_{q^2}(1)\,\vartheta_{q^6}(1)\,\vartheta_{q^{12}}(1)+q\,\Bigl[\vartheta_{q^2}(q)\,\vartheta_{q^6}(q^{3})\,\vartheta_{q^{12}}(1)\\
 &+2\Bigl(\vartheta_{q^2}(1)\,\vartheta_{q^6}(q^2)+\vartheta_{q^2}(q)\,\vartheta_{q^6}(q)\Bigr)\vartheta_{q^{12}}(q^4)\Bigr]\,.  
 \end{split}
\end{equation}

It may be noted that both \eqref{equation:cpas40} and \eqref{equation:c40} are much more
complicated that the formulas found in Corollary \ref{cor:deuxiemeidentite}.



\section{Asymptotics of the powers of theta and vanishing of the coefficients}
\label{section:asymptoticsandvanishing}


\subsection{Some sufficient conditions for the vanishing of $\gamma_{k,n}(q)$}
\label{subsection:CSAnnulation}


\subsubsection{Main statement and reduction to an asymptotic estimation}

Recall that, for $k \geq 1$, underlined characters denote $k$-uples:
$\x = (x_1,\ldots,x_k)$. \\

We first introduce a linear and a quadratic form on $\R^k$:
$$
S(\x) := \sum_{1 \leq i \leq k} x_i, \quad
Q(\x) := \sum_{1 \leq i \leq j \leq k} x_i x_j =
\dfrac{1}{2}\left(\sum_{1 \leq i \leq k} x_i^2 + S(\x)^2\right).
$$
The last equality shows that $Q$ is a positive definite form.

\begin{lem}
\label{lem:tnk(q)=fnk(x)}
Let $k \in \N^*$ and $n \in \Z$. Put, for $x \in \C$, $\lmod x \rmod < 1$:
$$
f_{k,n}(x) := \sum_{\m \in \Z^k} x^{Q(\m) - n S(\m)}.
$$
Then:
\begin{equation}
\label{eqn:tnk(q)=fnk(x)}
\forall q \in \C \setminus \Df(0,1) \;,\; \gamma_{k+1,n}(q) = q^{n(n-1)/2} f_{k,n}(q).
\end{equation}
The function $f_{k,n}$ is holomorphic over the punctured open disk $\Do(0,1) \setminus \{0\}$
with a pole of order $n(n-1)/2$ at $0$ (meaning that if $n \in \{0,1\}$ then $f_{k,n}$ is
actually holomorphic at $0$).
\end{lem}
\begin{proof}
From formula \eqref{eqn:formuleexactecoeffstheta} one draws the following calculation:
\begin{align*}
\gamma_{k+1,n}(q)
&= \sum_{m_0,\ldots,m_k \in \Z \atop m_0 + \cdots + m_k = n}
q^{\frac{m_0^2 + \cdots + m_k^2 - m_0 - \cdots - m_k}{2}} \\
&= \sum_{m_1,\ldots,m_k \in \Z}
q^{\frac{(m_1 + \cdots + m_k - n)^2 + m_1^2 + \cdots + m_k^2 - n}{2}} \\
&= \sum_{m_1,\ldots,m_k \in \Z}
q^{\frac{(m_1 + \cdots + m_k)^2 - 2 n (m_1 + \cdots + m_k) + n^2 + m_1^2 + \cdots + m_k^2 - n}{2}} \\
&= q^{n(n-1)/2} \sum_{m_1,\ldots,m_k \in \Z}
q^{\frac{(m_1 + \cdots + m_k)^2 - 2 n (m_1 + \cdots + m_k) + m_1^2 + \cdots + m_k^2}{2}} \\
&= q^{n(n-1)/2} \sum_{\m \in \Z^k} q^{Q(\m) - n S(\m)}.
\end{align*}
\end{proof}

\begin{thm}
\label{thm:estimasymptfnk}
Let $k \in \N^*$ and $n \in \Z$. Writing $\epsilon := (-1)^n$, one has, when $x \to 1^-$:
\begin{equation}
\label{eqn:estimasymptfnk}
f_{k,n}(-x) = \dfrac{\pi^{\frac{k}{2}} \sqrt{2}}{\sqrt{k+1}} \,
\dfrac{\cos((k + \epsilon)\pi/4)}{(1-x)^{\frac{k}{2}}}
+ O\left(\dfrac{1}{(1-x)^{\frac{k-1}{2}}}\right).
\end{equation}
\end{thm}

From \ref{subsubsection:deuxpropaux} on, most of the rest\footnote{Note however that
in \ref{subsubsection:tentative} we apply modular properties of generating series of
number of representations by quadratic forms to obtain some complementary informations
(proposition \ref{prop:CScomplementaire}).} of subsection \ref{subsection:CSAnnulation}
is devoted to the proof of theorem \ref{thm:estimasymptfnk}.

\paragraph{Consequences of theorem \ref{thm:estimasymptfnk}.}

We first discuss here consequences of this theorem of interest for our problem
(possibility of vanishing of the coefficients $\gamma_{k,n}$). Assume $k \geq 2$
and $0 \leq n \leq k$, so that $\gamma_{k,n}(0) = {k \choose n} > 0$ by formula
\eqref{eqn:DASgamma(q)en0} page \pageref{eqn:DASgamma(q)en0}. On the other hand,
the sign of $\gamma_{k,n}(q)$ as $q \to -1$ while staying within $\left]-1,0\right[$
is (after lemma \ref{eqn:tnk(q)=fnk(x)}) $(-1)^{n(n-1)/2} \times \cos k' \pi/4$, where
$k' := k - 1 + (-1)^n$; so this sign depends on $k \pmod 8 \in \{0,\ldots,7\}$
and on $n \pmod{4} \in \{0,\ldots,3\}$ (see our notational conventions at the
very end of section \ref{section:introduction}). More precisely:
$$
(-1)^{n(n-1)/2} = \begin{cases} +1 \text{~if~} n \!\!\!\!\!\! \pmod{4} \in \{0,1\}, \\
  -1 \text{~if~} n \!\!\!\!\!\! \pmod{4} \in \{2,3\}, \end{cases}
$$
and
$$
\text{~sign~} \cos k' \pi/4 = \begin{cases}
+1 \text{~if~} k' \!\!\!\!\!\! \pmod{8} \in \{7,0,1\}, \\  
~~~0 \text{~if~} k' \!\!\!\!\!\! \pmod{8} \in \{2,6\}, \\  
-1 \text{~if~} k' \!\!\!\!\!\! \pmod{8} \in \{3,4,5\}.
\end{cases}
$$
Now, $k' = k$ if $n$ is even and $k-2$ if $n$ is odd, so that we list the cases when
$\gamma_{k,n}(q) < 0$ as $q \to -1$
, we know that in
all these cases it must vanish somewhere.

\begin{cor}
\label{cor:CSAnnulation}
Assume $k \geq 2$ and $0 \leq n \leq k$. In all the following cases, $\gamma_{k,n}(q)$
vanishes for at least one value of $q \in \left]-1,0\right[$:
\begin{itemize}
\item $n \!\! \pmod{4} = 0$ and $k \!\! \pmod{8} \in \{3,4,5\}$,
\item $n \!\! \pmod{4} = 1$ and $k \!\! \pmod{8} \in \{5,6,7\}$,
\item $n \!\! \pmod{4} = 2$ and $k \!\! \pmod{8} \in \{7,0,1\}$,
\item $n \!\! \pmod{4} = 3$ and $k \!\! \pmod{8} \in \{1,2,3\}$,
\end{itemize}
All these conditions (\ie\ their logical disjunction) can be summarized in a unique one:
$$
k - 2n \!\!\!\! \pmod{8} \in \{3,4,5\}.
$$
\end{cor}

Note that the above conditions exclude \emph{a priori} the case $k = 2$ (because for
$n = 0,1,2$ one does not have $k - 2n \!\!\!\! \pmod{8} \in \{3,4,5\}$), so in what
follows we assume that $k \geq 3$ and write:
$$
Z_k := \{n \in \Z \tq \gamma_{k,n}(q) \text{~vanishes somewhere in~} \left]-1,0\right[\}.
$$
By the remark at the end of \ref{subsubsection:linearrelations}, $Z_k$ is $k$-periodic
\ie\ $n \in Z_k \Rightarrow n + k \Z \subset Z_k$. In order to use that fact, we introduce
some more notations:
\begin{align*}
X_k &:= \{n \in \{0,\ldots,k\} \tq k - 2n \!\!\!\! \pmod{8} \in \{3,4,5\}\}, \\
Y_k &:= X_k + k \Z.
\end{align*}

Note that $X_k$ is invariant under the symmetry $n \leftrightarrow k - n$, and therefore
so is $Y_k$: we shall not have to exploit further the symmetry expressed by equation
\eqref{eqn:symetriegamma}, it is already built-in. With the previous notations, we
thus have:

\begin{cor}
\label{cor:CSAnnulationMieux}
Let $k \geq 3$. Then $Y_k \subset Z_k$.
\end{cor}

To make this more useful, we describe more precisely $X_k$, and, in some cases, $Y_k$.

\paragraph{Case $k$ even: $k = 8 a + 2 b$, $b \in \{0,1,2,3\}$.}

Then in condition $k - 2n \!\!\!\! \pmod{8} \in \{3,4,5\}$ only the remainder $4$
is possible:
$$
X_k = \{n \in \{0,\ldots,k\} \tq k \equiv 2n + 4 \pmod{8}\} =
\{n \in \{0,\ldots,k\} \tq n \equiv b + 2 \pmod{4}\}.
$$
For instance, one easily checks the following cases:
\begin{itemize}
\item If $k = 8 a$, then $X_k = 4 \, \{0,\ldots,2a\} + 2$ and $Y_k = 4 \Z + 2$.
\item If $k = 8 a + 4$, then $X_k = 4 \, \{0,\ldots,2a\}$ and $Y_k = 4 \Z$.
\end{itemize}

\begin{cor}
All $\gamma_{8a,4c+2}(q)$ and all $\gamma_{8a+4,4c}(q)$ vanish for some $q \in ]-1,0[$.
\end{cor}

Cases $k = 8a + 2$ and $k = 8 a + 6$ can be similarly described (but less simply).

\paragraph{Case $k$ odd: $k = 8 a + 2 b + 1$, $b \in \{0,1,2,3\}$.}

Then in condition $k - 2n \!\!\!\! \pmod{8} \in \{3,4,5\}$ only cases $3$ and
$5 \equiv -3$ are possible:
$$
X_k = \{n \in \{0,\ldots,k\} \tq k \equiv 2n \pm 3 \pmod{8}\} =
\{n \in \{0,\ldots,k\} \tq n \equiv b + 2 \text{~or~} b - 1 \pmod{4}\}.
$$
Again these cases can be similarly but less simply described. We just give an example:
if $k = 3$, we find that $X_3 = \{0,2\}$ and $Y_3 = 3 \Z \pm 1$.


\subsubsection{Two auxiliary propositions and an application to a generating series}
\label{subsubsection:deuxpropaux}

\paragraph{First auxiliary proposition.}

Let $L(\y) := \sum\limits_{1 \leq i \leq k} b_i y_i$ a linear form and
$F(\y) := \sum\limits_{1 \leq i \leq j \leq k} a_{i,j} y_i y_j$ a positive
definite quadratic form on $\R^k$. Write, for $M > 0$:
$$
E_M := \{\y \in \R^k \tq F(\y) + L(\y) \leq M\}
\quad \text{and} \quad \Lambda_M := E_M \cap \Z^k.
$$
Since $F$ is positive definite, $E_M$ is a compact ellipsoid and $\Lambda_M$
is a finite set. \\

The volume of $E_M$ is easily computed as follows. Let $V_k$ the volume of the
unit ball of $\R^k$ and let $D_F$ the discriminant of $F$. The matrix of $F$
writes $^{t} S S$, where $S$ is invertible, and $D_F = (\det S)^2$. Therefore
$F(\y) = (S \y)^2 := \langle S \y,S \y \rangle$. In the same way, $L(\y) = B \y$
for some line matrix $B$. If $\y_0 := \dfrac{1}{2} {}^t (B S^{-1})$,
one sees that $\langle S \y + \y_0,S \y + \y_0 \rangle^2 = F(\y) + L(\y) + \y_0^2$.
Thus, the mapping $\y \mapsto S \y + \y_0$ transforms $E_M$ into a ball with radius
$\sqrt{M + \y_0^2}$, the volume of which is $V_k (M + \y_0^2)^{k/2}$. On the other
hand, this affine linear mapping multiplies all volumes by
$\lmod \det S \rmod = D_F^{1/2}$. One concludes:
$$
\mu(E_M) = \dfrac{V_k}{D_F^{1/2}} (M + \y_0^2)^{k/2}.
$$

\begin{prop}
\label{prop:cardinalLambdaM}
Let $V_k$ the volume of the unit ball of $\R^k$ and let $D_F$ the discriminant
of $F$. Then, when $M \to + \infty$:
\begin{equation}
\label{eqn:cardinalLambdaM}
\card\ \Lambda_M = \dfrac{V_k}{D_F^{1/2}} M^{k/2} + O\left(M^{(k-1)/2}\right).
\end{equation}
\end{prop}
\begin{proof}
Let $E'_M := \Lambda_M + \left[-1/2,1/2\right]^k$, the volume $\mu(E'_M)$ of
which is $\card\ \Lambda_M$. We shall first determine $M',M'' > 0$ such that
$E_{M'} \subset E'_M \subset E_{M''}$, which will provide us uper and lower bound
for $\card\ \Lambda_M = \mu(E'_M)$. \\
Let $\y \in \Lambda_M$ and $\u \in \left[-1/2,1/2\right]^k$ and let us write
$\phi := F + L$. Triangle inequalities applied to the norm $\sqrt{F}$ imply:
$$
\sqrt{F(\y)} - a \leq \sqrt{F(\y)} - \sqrt{F(\u)} \leq \sqrt{F(\y + u)} \leq
\sqrt{F(\y)} + \sqrt{F(\u)} \leq \sqrt{F(\y)} + a,
$$
where $a := \dfrac{1}{2} \sqrt{\sum\limits_{1 \leq i \leq j \leq k} \lmod a_{i,j} \rmod}$,
because $F(\u) \leq a^2$ for every $\u \in \left[-1/2,1/2\right]^k$. One deduces:
$$
F(\y) - 2 a \sqrt{F(\y)} + a^2 \leq F(\y + \u) \leq F(\y) + 2 a \sqrt{F(\y)} + a^2.
$$
In the same way, $L(\y) - b \leq L(\y + \u) = L(\y) + L(\u) \leq L(\y) + b$,
where $b := \dfrac{1}{2} \sum\limits_{1 \leq i \leq k} \lmod b_i \rmod$. One draws:
$$
\phi(\y) + a^2 - b - 2 a \sqrt{F(\y)} \leq \phi(\y + \u) \leq
\phi(\y) + a^2 + b + 2 a \sqrt{F(\y)}.
$$
On the other hand, there exists a constant $c > 0$ such that
$\lmod L \rmod \leq c \sqrt{F}$ (continuity of the linear form $L$ with respect
to the topology defined by the norm $\sqrt{F}$, all this in finite dimension).
One deduces:
$$
\phi(\y) \geq F(\y) - c \sqrt{F(\y)} = (\sqrt{F(\y)} - c/2)^2 - c^2/4
\Longrightarrow
\sqrt{F(\y)} \leq c/2 + \sqrt{\phi(\y) + c^2/4}.
$$
In the end, we find, for $d' := a^2 - b - ac$ and $d'' := a^2 + b + ac$,
the bounds:
$$
\phi(\y) - 2 a \sqrt{\phi(\y) + c^2/4} + d' \leq \phi(\y + \u) \leq
\phi(\y) + 2 a \sqrt{\phi(\y) + c^2/4} + d''
$$
valid for all $\y \in \Lambda_M$ and $\u \in \left[-1/2,1/2\right]^k$. It is then
obvious that putting:
$$
M' := M - 2 a \sqrt{M + c^2/4} + d' \text{~and~} M'' := M + 2 a \sqrt{M + c^2/4} + d''
$$
one has indeed:
$$
E_{M'} \subset E'_M \subset E_{M''} \Longrightarrow
\mu(E_{M'}) \leq \card\ \Lambda_M \leq \mu(E_{M''}),
$$
said otherwise:
$$
\dfrac{V_k}{D_F^{1/2}} (M - 2 a \sqrt{M + c^2/4} + d' + \y_0^2)^{k/2}
\leq \card\ \Lambda_M \leq
\dfrac{V_k}{D_F^{1/2}} (M + 2 a \sqrt{M + c^2/4} + d'' + \y_0^2)^{k/2}.
$$
Since, for all $\alpha,\beta,\gamma \in \R$, one has:
$$
(M + \alpha \sqrt{M + \beta} + \gamma)^{\frac{k}{2}} =
M^{\frac{k}{2}} \left(1 + O\left(M^{\frac{-1}{2}}\right)\right)^{\frac{k}{2}} =
M^{\frac{k}{2}} + O\left(M^{\frac{k-1}{2}}\right),
$$
the desired estimation follows.
\end{proof}

\paragraph{Second auxiliary proposition.}

\begin{prop}
\label{prop:sgpuissancesden}
For any $\alpha > 0$ and $x \in \C$, $\lmod x \rmod < 1$, let
$S_\alpha(x) := \sum\limits_{n \geq 0} n^\alpha x^n$. Then, when $x \to 1^-$:
\begin{equation}
\label{eqn:sgpuissancesden}
S_\alpha(x) = \dfrac{\Gamma(\alpha+1)}{(1-x)^{\alpha+1}} +
O\left(\dfrac{1}{(1-x)^\alpha}\right).
\end{equation}
\end{prop}
\begin{proof}
Write, for $0 < x < 1$, $a := - \ln x > 0$ and, for $t \geq 0$,
$h_a(t) := t^\alpha x^t = t^\alpha e^{-a t}$. We thus have:
$$
I := \int_0^{+\infty} h_a(t) \, dt  = \dfrac{\Gamma(\alpha+1)}{a^{\alpha+1}} \cdot
$$
The function $h_a$ grows over $\left[0,\dfrac{\alpha}{a}\right]$ and
decreases (with limit $0^+$) over $\left[\dfrac{\alpha}{a},+\infty\right[$;
its maximum is
$h_a\left(\dfrac{\alpha}{a}\right) = \left(\dfrac{\alpha}{ae}\right)^\alpha$. \\
If $\alpha = a/m$, $m\in \N^*$, whence $\dfrac{\alpha}{a} = m$, one deduces the
bounds:
\begin{eqnarray*}
h_a(0) + \cdots + h_a(m-1) & \leq \int_0^{m} h_a(t) \, dt
& \leq h_a(1) + \cdots + h_a(m) \\
h_a(m+1) + \cdots + h_a(n) + \cdots & \leq \int_m^{+\infty} h_a(t) \, dt
& \leq h_a(m) + \cdots + h_a(n) + \cdots ,
\end{eqnarray*}
whence $\sum\limits_{n \geq 0} h_a(n) \in \left[I - h_a(m),I + h_a(m)\right]$. Therefore:
$$
S_\alpha(x) =
\dfrac{\Gamma(\alpha+1)}{(-\ln x)^{\alpha+1}} + O\left(\dfrac{1}{(- \ln x)^\alpha}\right) =
\dfrac{\Gamma(\alpha+1)}{(1-x)^{\alpha+1}} + O\left(\dfrac{1}{(1-x)^\alpha}\right)
$$
for $\ln x \sim x - 1$ when $x \to 1$. The estimation is therefore right for all
$a := \alpha/m$, but since $S_\alpha(x)$ is monotonous as a function of $a$,
it stays valid without restriction when $a \to 0^+$ and $x \to 1^-$.
\end{proof}

\paragraph{Application to a generating series.}

Let $F$ and $L$ as defined at the beginning of \ref{subsubsection:deuxpropaux} and
define the generating series:
\begin{equation}
\label{eqn:defPhi_F,L}
\Phi_{F,L}(x) := \sum_{\m \in \Z^k} x^{F(\m) + L(\m)} = \sum_{n \in _Z} r_{F,L}(n) x^n,
\end{equation}
where $r_{F,L}(n) := \card\ \{\m \in \Z^k \tq F(\m) + L(\m) = n\}$. The series
$\Phi_{F,L}(x)$ only has a finite number of terms with negative exponent and it
converges normally over every compact subset of the punctured open disk
$\Do(0,1) \setminus \{0\}$.

\begin{prop}
\label{prop:estimationPhiFL}
For $x \to 1^-$, one has the asymptotic estimation:
\begin{equation}
\label{eqn:estimationPhiFL}
\Phi_{F,L}(x) = \dfrac{1}{D_F^{1/2}} \dfrac{\pi^{\frac{k}{2}}}{(1-x)^{\frac{k}{2}}}
+ O\left(\dfrac{1}{(1-x)^{\frac{k-1}{2}}}\right).
\end{equation}
\end{prop}
\begin{proof}
We abreviate $C := \dfrac{V_k}{D_F^{1/2}} \cdot$. Then:
$$
\dfrac{\Phi_{F,L}(x)}{1-x} =
\left(\sum_{n \in \Z} r_{F,L}(n) x^n\right)\left(\sum_{n \in \Z} x^n\right) =
\sum_{n \in \Z} R_{F,L}(n) x^n,
\text{~or~} R_{F,L}(n) := \sum_{p \leq n} r_{F,L}(p) = \card\ \Lambda_n
$$
with the same notations as before. From proposition \ref{prop:cardinalLambdaM},
there is a constant $C' > 0$ such that:
$$
\forall n > 0 \;,\; \lmod R_{F,L}(n) - C n^{\frac{k}{2}} \rmod < C' n^{\frac{k-1}{2}}.
$$
Let $\Psi(x) := \sum\limits_{n > 0} R_{F,L}(n) x^n$, so that, when $x \to 1^-$:
$$
\dfrac{\Phi_{F,L}(x)}{1-x} = \Psi(x) + O(1).
$$
We immediately get, for $0 < x < 1$:
$$
\lmod \Psi(x) - C S_{\frac{k}{2}}(x) \rmod \leq C' S_{\frac{k-1}{2}}(x).
$$
But, after proposition \ref{prop:sgpuissancesden}:
$$
S_{\frac{k}{2}}(x) = \dfrac{\Gamma(\frac{k}{2}+1)}{(1-x)^{\frac{k}{2}+1}} +
O\left(\dfrac{1}{(1-x)^{\frac{k}{2}}}\right)
\text{~and~}
S_{\frac{k-1}{2}}(x) = \dfrac{\Gamma(\frac{k-1}{2}+1)}{(1-x)^{\frac{k-1}{2}+1}} +
O\left(\dfrac{1}{(1-x)^{\frac{k-1}{2}}}\right),
$$
whence:
$$
\dfrac{\Phi_{F,L}(x)}{1-x} = C \dfrac{\Gamma(\frac{k}{2}+1)}{(1-x)^{\frac{k}{2}+1}}
+ O\left(\dfrac{1}{(1-x)^{\frac{k+1}{2}}}\right).
$$
Since we know that $V_k = \dfrac{\pi^{\frac{k}{2}}}{\Gamma(\frac{k}{2}+1)}$, the
desired estimation follows.
\end{proof}


\subsubsection{End of the proof of theorem \ref{thm:estimasymptfnk}}

\paragraph{Odd and even parts of $f_{k,n}$.}

By partitioning $\Z^k$ into $2^k$ classes modulo the subgroup $(2 \Z)^k$,
one decomposes the generating series $f_{k,n}$:
\begin{align*}
f_{k,n}(x)
&= \sum_{\m \in \Z^k} x^{Q(\m) - n S(\m)} \\
&= \sum_{\es \in \{0,1\}^k} \sum_{\m \in \Z^k} x^{Q(2 \m + \es) - n S(2 \m + \es)} \\
&= \sum_{\es \in \{0,1\}^k} x^{Q(\es) - n S(\es)}
\sum_{\m \in \Z^k} x^{4 Q(\m) + 4 B(\m,\es) - 2 n S(\m)},
\end{align*}
where $B$ denotes the symmetric bilinear form such that $B(\m,\m) = Q(\m)$.
Each sum $\sum\limits_{\m \in \Z^k} x^{4 Q(\m) + 4 B(\m,\es) - 2 n S(\m)}$ is even as
a function of $\m$, because of the form of the exponents:
$$
4 Q(\m) + 4 B(\m,\es) - 2 n S(\m) = 4 \sum_{1 \leq i \leq j \leq k} m_i m_j
+ 2 \sum_{1 \leq i \leq j \leq k} (m_i \epsilon_j + m_j \epsilon_i)
- 2 n \sum_{1 \leq i \leq k} m_i.
$$
Therefore:
\begin{align*}
f_{k,n}(-x)
&= \sum_{\es \in \{0,1\}^k} (-1)^{Q(\es) - n S(\es)} x^{Q(\es) - n S(\es)} 
\sum_{\m \in \Z^k} x^{4 Q(\m) + 4 B(\m,\es) - 2 n S(\m)} \\
&= \sum_{\es \in \{0,1\}^k} (-1)^{Q(\es) - n S(\es)} x^{Q(\es) - n S(\es)} \Phi_{Q,L_\es}(x^4),
\end{align*}
where we use again the generating series $\Phi_{F,L}$ of \eqref{eqn:defPhi_F,L},
here with the quadratic form $F := Q$ and the linear form $L := L_\es$ defined by:
$$
L_\es(\m) := B(\m,\es) - (n/2) S(\m).
$$

\begin{lem}
\label{lem:estimationPhiQLe}
When $x \to 1^-$, we have:
\begin{equation}  
\label{eqn:estimationPhiQLe}
f(-x) = \left(\sum_{\es \in \{0,1\}^k} (-1)^{Q(\es) - n S(\es)}\right) 
\dfrac{1}{2^k D_Q^{1/2}} \dfrac{\pi^{\frac{k}{2}}}{(1-x)^{\frac{k}{2}}}
+ O\left(\dfrac{1}{(1-x)^{\frac{k-1}{2}}}\right).
\end{equation}
\end{lem}
\begin{proof}
One calls upon proposition \ref{prop:estimationPhiFL} and the fact that
$\dfrac{1}{(1 - x^4)^{\frac{k}{2}}} \sim \dfrac{1}{2^k (1 - x)^{\frac{k}{2}}}$
when $x \to 1$.
\end{proof}

We are left with the task of determining the value (and the sign !) of the factor
$\sum\limits_{\es \in \{0,1\}^k} (-1)^{Q(\es) - n S(\es)}$.

\paragraph{Calculation of the factor
  $\sum\limits_{\es \in \{0,1\}^k} (-1)^{Q(\es) - n S(\es)}$
  and conclusion.}

Let
$$
A := \card \{\es \in \{0,1\}^k \tq Q(\es) \equiv n S(\es) \pmod{2}\}
\text{~and~}
B := \card \{\es \in \{0,1\}^k \tq Q(\es) \not\equiv n S(\es) \pmod{2}\}.
$$
One thus has:
$$
A - B = \sum\limits_{\es \in \{0,1\}^k} (-1)^{Q(\es) - n S(\es)}
\quad \text{and} \quad A + B = 2^k.
$$

If $l$ among the $k$ components $\epsilon_j$ of $\es$ have value $1$ (so that
$k-l$ have value $0$), then $Q(\es) = l(l+1)/2$ and $S(\es) = l$, whence:
$$
A = \sum_{0 \leq l \leq k \atop \frac{l(l+1)}{2} - n l \text{~even}} {k \choose l}
\quad \text{~and~} \quad
B = \sum_{0 \leq l \leq k \atop \frac{l(l+1)}{2} - n l \text{~odd}} {k \choose l}
$$
We thus must discuss the parity of $l(l+1)/2 - n l$. If $n$ is even, it is the
same as that of $l(l+1)/2$, which is even if, and only if, the remainder of $l$
modulo $4$ is $0$ or $3$; if $n$ is odd, it is the same as that of $l(l-1)/2$,
which is even if, and only if the remainder of $l$ modulo $4$ is $0$ or $1$.
We are led to introduce:
$$
\forall j = 0,1,2,3 \;,\; s_j := \sum_{0 \leq l \leq k \atop l \equiv j \pmod{4}} {k \choose l}.
$$
We then have:
$$
(A,B) = \begin{cases} (s_0 + s_3,s_1 + s_2) \text{~if~} n \text{~is even}, \\
(s_0 + s_1,s_2 + s_3) \text{~if~} n \text{~is odd}. \end{cases}
$$
The integers $s_j$ can be deduced from formulas (given here for $k > 0$):
\begin{eqnarray*}
s_0 + s_1 + s_2 + s_3 &= (1+1)^k &= 2^k, \\
s_0 + \ii s_1 - s_2 - \ii s_3 &= (1 + \ii)^k &= 2^{k/2} (\cos k \pi/4 + \ii \sin k \pi/4), \\
s_0 - s_1 + s_2 - s_3 &= (1-1)^k &= 0, \\
s_0 - \ii s_1 - s_2 + \ii s_3 &= (1 - \ii)^k &= 2^{k/2} (\cos k \pi/4 - \ii \sin k \pi/4).
\end{eqnarray*}
One draws:
\begin{align*}
s_0 &= \dfrac{1}{4} \left(2^k + 2^{\frac{k}{2} + 1} \cos k \pi/4\right), \\
s_1 &= \dfrac{1}{4} \left(2^k + 2^{\frac{k}{2} + 1} \sin k \pi/4\right), \\
s_2 &= \dfrac{1}{4} \left(2^k - 2^{\frac{k}{2} + 1} \cos k \pi/4\right), \\
s_3 &= \dfrac{1}{4} \left(2^k - 2^{\frac{k}{2} + 1} \sin k \pi/4\right),
\end{align*}
whence, if $n$ is even:
$$
A - B = s_0 - s_1 - s_2 + s_3 =
\dfrac{1}{4} 2^{\frac{k}{2} + 1} (2 \cos k \pi/4 - 2 \sin k \pi/4) =
2^{\frac{k+1}{2}} \cos (k+1) \pi/4,
$$
and, if $n$ is odd:
$$
A - B = s_0 + s_1 - s_2 - s_3 =
\dfrac{1}{4} 2^{\frac{k}{2} + 1} (2 \cos k \pi/4 + 2 \sin k \pi/4) =
2^{\frac{k+1}{2}} \cos (k-1) \pi/4.
$$
Combining these values with lemma \ref{lem:estimationPhiQLe}, one finds:
$$
f_{k,n}(-x) = \dfrac{\pi^{\frac{k}{2}}}{2^{\frac{k-1}{2}} D_Q^{\frac{1}{2}}} \,
\dfrac{\cos((k + \epsilon)\pi/4)}{(1-x)^{\frac{k}{2}}}
+ O\left(\dfrac{1}{(1-x)^{\frac{k-1}{2}}}\right).
$$
This formula can be made more precise by noting that $D_Q = (k+1) 2^{-k}$.
(Quick argument: $a \in \C$ being fixed, let $D_n(x)$ the determinant of the
$n \times n$ matrix having $x$ on the diagonal and $a$ elsewhere. One checks
that $D_n'(x) = n D_{n-1}(x)$ and $D_n(a) = 0$, from which one deduces by
induction that $D_n(x) = (x-a)^{n-1} (x + (n-1)a)$. One then sets $x := 1$ and
$a := 1/2$.) Replacing $D_Q$ by $(k+1) 2^{-k}$ in the above formula:
$$
f_{k,n}(-x) = \dfrac{\pi^{\frac{k}{2}} \sqrt{2}}{\sqrt{k+1}} \,
\dfrac{\cos((k + \epsilon)\pi/4)}{(1-x)^{\frac{k}{2}}}
+ O\left(\dfrac{1}{(1-x)^{\frac{k-1}{2}}}\right).
$$
This is indeed formula \eqref{eqn:estimasymptfnk} of theorem
\ref{thm:estimasymptfnk}.


\subsubsection{Complement: a first attempt at exploiting modularity}
\label{subsubsection:tentative}
The generating series of the numbers of representations by positive definite
quadratic form are known to have modular properties, it has been, since Jacobi,
one of the clues of the appearance of theta functions in number theory. We shall
here try (and only partially succeed) to exploit this fact to improve the above
result.

\paragraph{Use of Poisson's summation formula.}

We use Poisson's formula in the form quoted in \cite[\S VII.6]{SerreCoursEN}.
Let $f$ a function over $\R^k$, of Schwartz class, meaning that $f$ is
$\mathcal{C}^\infty$ and all its partial derivatives $Df$ at all orders have
fast decay, \ie\ they are such that
$\forall N \;,\; \lmod Df(\x) \rmod = O\left(\lmod \x \rmod^{-N}\right)$
when $\lmod \x \rmod \to + \infty$. We write $\mu$ the Lebesgue measure
and $\x.\y$ the canonical inner product on $\R^k$. The Fourier transform
$\hat{f}$, defined as:
$$
\hat{f}(\y) := \int_{\R^k} e^{- 2 \ii \pi \x.\y} \, f(\x) \, d\mu(\x)
$$
also is of Schwartz class; and \emph{Poisson's summation formula} can be
stated as follows:
$$
\sum_{\m \in \Z^k} \hat{f}(\m) = \sum_{\m \in \Z^k} f(\m).
$$

Now let $F(\x) := \sum a_{i,j} x_i x_j$ a positive definite quadratic form and
$L(\x) := \sum b_i x_i$ a linear form over $\R^k$. The function:
$$
\phi_{F,L}(\x) := e^{-(F(\x) + L(\x))}
$$
is of Schwartz class. We shall compute its Fourier transform. We write
$X,Y,U,V,B,\ldots$ the column vectors associated with $\x,\y,\u,\v,\b\ldots$
and $A := (a_{i,j})$ the (real symmetric) matrix of $F$, so that
$F(\x) = \tr X A X$. Likewise $L(\x) = \tr B X$, where $B$ is the column vector
associated with $\b := (b_1,\ldots,b_k)$. Let $S$ a real square invertible matrix
such that $A = \tr S S$ (such $S$ is known to exist) and let $T := S^{-1}$.
Then:
\begin{align*}
\widehat{\phi_{F,L}}(\y)
&= \int_{\R^k} e^{- 2 \ii \pi \x.\y - (F(\x) + L(\x))} \, d\mu(\x) =
\int_{\R^k} e^{- \tr X A X - \tr B X - 2 \ii \pi \tr Y X} \, d\mu(\x) \\
&= \dfrac{1}{\lmod \det S \rmod}
\int_{\R^k} e^{- \tr U U - \tr B T U - 2 \ii \pi \tr Y T U} \, d\mu(\u)
\quad \text{by change of variable~} X = S U \\
&= \dfrac{1}{\lmod \det S \rmod}
\int_{\R^k} e^{- \u^2 - \v \u} \, d\mu(\u)
\quad \text{where one put~} \tr V := \tr B T + 2 \ii \pi \tr Y T \\
&= \dfrac{e^{\v^2/4}}{\lmod \det S \rmod} \int_{\R^k} e^{- (\u + \v/2)^2} \, d\mu(\u) =
\dfrac{e^{\v^2/4}}{\lmod \det S \rmod} \int_{\R^k} e^{-\u^2} \, d\mu(\u) =
\dfrac{\pi^{k/2} e^{\v^2/4}}{\lmod \det S \rmod} \cdot
\end{align*}
(Because $\int_{\R^k} e^{-\u^2} \, d\mu(\u)$ is the gaussian integral over $\R^k$.)
At last, we compute:
$$
\v^2 = (\tr B T + 2 \ii \pi \tr Y T)(\tr T B + 2 \ii \pi \tr T Y) =
(\tr B + 2 \ii \pi \tr Y) A^{-1} (B + 2 \ii \pi Y) = \tilde{F}(\b + 2 \ii \pi \y),
$$
where was introduced the quadratic form $\tilde{F}$ with matrix $A^{-1}$ (this form
also is positive definite). One has:
$$
\tilde{F}(\b + 2 \ii \pi \y) = \tilde{F}(\b)
- 4 \pi^2 \left(\tilde{F}(\y) + \tilde{L}(\y)\right), \text{~where~}
\tilde{L}(\y) := \dfrac{1}{\ii \pi} \tr B A^{-1} Y.
$$
And we eventually reaches the result:
$$
\widehat{\phi_{F,L}}(\y) =
\dfrac{\pi^{k/2} e^{\tilde{F}(\b/2)}}{\sqrt{\det A}} e^{-\pi^2(\tilde{F}(\y) + \tilde{L}(\y))}.
$$
Poisson's summation formula then yields:
\begin{equation}
\label{eqn:FSP}
\sum_{\m \in \Z^k} e^{-(F(\m) + L(\m))} = \dfrac{\pi^{k/2} e^{\tilde{F}(\b/2)}}{\sqrt{\det A}} 
\sum_{\m \in \Z^k} e^{-\pi^2(\tilde{F}(\m) + \tilde{L}(\m))}.
\end{equation}

\paragraph{Application to a generating series.}

Recall that $f_{k,n}(x) = \sum\limits_{\m \in \Z^k} x^{Q(\m) - n S(\m)}$ and that we are
looking for conditions ensuring the vanishing of this function. We have:
$$
f_{k,n}(x) = \sum_{\m \in \Z^k} \phi_{F,L}(\m), \text{~where~}
F := t Q, L := - t n S, \text{~with~} x = e^{-t}. 
$$
Note that $\lmod x \rmod < 1 \Leftrightarrow \Re(t) > 0$.
In order to apply formula \ref{eqn:FSP}, we thus take:
$$
A := t \begin{pmatrix} 1 & 1/2 & \ldots & 1/2 \\ 1/2 & 1 & \ldots & 1/2 \\
\vdots & \vdots & \ddots & \vdots \\ 1/2 & 1/2 & \ldots & 1 \end{pmatrix}
\quad \text{and} \quad
B := - t n \begin{pmatrix} 1 \\ 1 \\ \vdots \\ 1 \end{pmatrix}.
$$
One easily checks that:
$$
A^{-1} = \dfrac{2}{(k+1)t}
\begin{pmatrix} k & -1/2 & \ldots & -1 \\ -1 & k & \ldots & -1 \\
\vdots & \vdots & \ddots & \vdots \\ -1 & -1 & \ldots & k \end{pmatrix},
$$
whence:
$$
\tilde{F}(\y) = \dfrac{2}{(k+1)t} \tilde{Q}(\y), \text{~where~}
\tilde{Q}(\y) := k \sum_{1 \leq i \leq k} y_i^2 - 2 \sum_{1 \leq i < j \leq k} y_i y_j,
\quad \text{and~} \tilde{L}(\y) = \dfrac{-2 n}{(k+1) \ii \pi} S(\y).
$$
Setting $y := e^{-\frac{2 \pi^2}{(k+1)t}}$, we then see that $f_{k,n}(x)$ equals,
up to a non zero constant factor:
$$
g(y) := \sum_{\m \in \Z^k} e^{r \ii \pi S(\m)} \, y^{\tilde{Q}(\m)},
$$
where one put $r := - \dfrac{2 n}{k+1} \in \Q$. The function $g(y)$ is well defined
for $\lmod y \rmod < 1$, which is consistent with the above definition of $y$ and the
fact that $\lmod x \rmod < 1 \Leftrightarrow \Re(t) > 0 \Leftrightarrow \Re(1/t) > 0$.
Note that, since $\tilde{Q}$ is positive definite, $g(0) = 1$. But the presence of a
\og phase \fg\ $r \ii S(\m)$ on the one hand, the arithmetical properties of $\tilde{Q}$
on the other hand, complicate slightly the use of real analysis here. \\

We follow, as far as possible, the same method as before:
\begin{align*}
g(y)
&= \sum_{\es \in \{0,1\}^k} \sum_{\m \in \Z^k} e^{r \ii \pi S(2 \m + \es)} \, y^{\tilde{Q}(2 \m + \es)} \\
&= \sum_{\es \in \{0,1\}^k} e^{r \ii \pi S(\es)} \, y^{\tilde{Q}(\es)}
\sum_{\m \in \Z^k} e^{2 r \ii \pi S(\m)} \, (y^4)^{\tilde{Q}(\m) + \tilde{B}(\m,\es)},
\end{align*}
where $\tilde{B}(\u,\v) := k \sum u_i v_i - 2 \sum\limits_{i<j} (u_i v_j + u_j v_i)$
denotes the symmetric bilinear form such that $\tilde{B}(\m,\m) = \tilde{Q}(\m)$.
One also finds out that, if $\es$ has $l$ coefficients with value $1$, then
$S(\es) = l$ and $\tilde{Q}(\es) = kl - l(l-1)$.

\paragraph{If $r$ is an integer.}

In that case, $e^{2 r \ii \pi S(\m)} = 1$, which somehow gets rid of phase problems.
One want to make $y^4 \to 1^-$ but in such a way that $g(y)$ stays real. We will
take $y = \ii^p z$, $0 \leq z < 1$, $p \in \{0,1,2,3\}$ being chosen such that
$pk$ be even: $pk = 2 s$. Then:
$$
y^{\tilde{Q}(\es)} = \ii^{p kl - pl(l-1)}  z^{\tilde{Q}(\es)} =
(-1)^{sl - pl(l-1)/2}  z^{\tilde{Q}(\es)},
$$
and all terms are indeed real. Moreover, we have the following asymptotic estimate
when $z \to 1^-$:
$$
g(y) =
\left(\sum_{l=0}^k {k \choose l} (-1)^{r l + sl - pl(l-1)/2}\right) \dfrac{C}{(1-z)^{k/2}}
+ O\left(\dfrac{1}{(1-z)^{(k-1)/2}}\right),
$$
where $C > 0$ is computed as before using the discriminant of $\tilde{Q}$. \\

If the left factor of the main term is strictly negative, the limit of $g(y)$ is
$-\infty$ and one deduces again the existence of a zero of $g(y)$. To compute said
factor, one notes that $l(l-1)/2$ is even if, and only if $l \equiv 0$ or
$l \equiv 1 \pmod{4}$. One then has, with the same notations as before:
\begin{align*}
\sum_{l=0}^k {k \choose l} (-1)^{r l + sl - pl(l-1)/2}
&= s_0 + (-1)^{s+r} s_1 + (-1)^p (s_2 + (-1)^{s+r} s_3) \\
&= (s_0 + (-1)^p s_2) + (-1)^{s+r} (s_1 + (-1)^p s_3).
\end{align*}
If $p$ is even, since $s_0 + s_2 = s_1 + s_3 = 2^{k-1}$, one finds $0$ or $2^k$ and
therefore certainly not a strictly negative number. So we suppose that $p$ is odd,
\ie\ $p = 1$ or $p = 3$. Either case amounts to the same, so we take $p = 1$ and
therefore $k = pk = 2s$. Since $r = -2 n/(k+1)$ has been assumed integral and since
the denominator is odd, we see that $n$ is a multiple of $k+1$ and that $r$ is even. \\

Let us resume the calculation:
\begin{align*}
\sum_{l=0}^k {k \choose l} (-1)^{r l + sl - pl(l-1)/2}
&= (s_0 - s_2) + (-1)^{s+r} (s_1 - s_3) \\
&= 2^{k/2} (\cos k \pi/4 + (-1)^{s+r} \sin k \pi/4) \\
&= 2^s (\cos s \pi/2 + (-1)^{s+r} \sin s \pi/2).
\end{align*}
This number is strictly negative if, and only if,
$s \equiv 1 \text{~or~} 2 \pmod{4}$. So we can conclude this study:

\begin{prop}
\label{prop:CScomplementaire}
If $k \equiv 2 \text{~or~} 4 \pmod{8}$ and if $k+1 | n$, then $\gamma_{k,n}(q)$ vanishes
for at least one value of $q$.
\end{prop}

It is not easy here to characterize that value: one would require that the argument
of $y := e^{-\frac{2 \pi^2}{(k+1)t}}$ be $\pi/2$, whence a condition on $t$ which we shall
not try to write down. 

\paragraph{If $r' := r-1/2$ is an integer.}

In that case, one checks that $k$ is necessarily odd: $k = 2 s - 1$. An analysis
similar to the previous one leads us to take $y = \ii z$, whence again $y^4 = z^4$.
With the same notations, one then gets:
$$
e^{r \ii \pi S(\es)} \, y^{\tilde{Q}(\es)} = (-1)^{r' + sl - l(l-1)/2} \, z^{l(2s - l)}
$$
and
$$
\sum_{\m \in \Z^k} e^{2 r \ii \pi S(\m)} \, (y^4)^{\tilde{Q}(\m) + \tilde{B}(\m,\es)} =
\sum_{\m \in \Z^k} (-1)^{S(\m)} \, (z^4)^{\tilde{Q}(\m) + \tilde{B}(\m,\es)}.
$$
It is the latter sum that causes problems: one can indeed check that the terms
with $S(\m)$ even, resp. odd, mutually compensate in such a way that the main
term of the asymptotic estimation disappears. Our method does not allow to
conclude in this case.


\subsection{Another approach to asymptotics when $q\to -1$}
\label{section:-1}

Here again and until the end of the paper, we shall use the function $\vartheta$
and the coefficients $c_{k,n}(q)$ introduced in \ref{subsection:Puissance4}.
We begin with a kind of ``distribution formula''. \\

Recall that $c_{1,0}(q)=1$. From now on, we shall assume that $k \ge 2$; also,
from equation \eqref{equation:k}:
$$
\vartheta^k(x)=\sum_{n=0}^{k-1}c_{k,n}(q)\,x^n\,\vartheta_{q^k}(x^kq^n)\,.
$$
Last, recall the notation $\sum\limits_{j \!\! \pmod{k}} f(j)$ which was introduced
at the very end of section \ref{section:introduction}.

\begin{prop}
\label{prop:ckl}
If $\mu$ is a primitive $k^\text{th}$ root of unity, one has: 
\begin{equation}
\label{equation:ck1}
c_{k,n}(q) =
\dfrac{q^{n^2/k}\,x^{-n}}{k\,\vartheta_{q^k}(x^k)}\,
\sum_{j \!\!\!\!\!\!\!\! \pmod{k}} \mu^{-n j}\,\vartheta^k(xq^{-n/k}\mu^j)\,.
\end{equation}
In particular:
\begin{equation}
\label{equation:ck2}
c_{k,n}(q) =
\dfrac{q^{n^2/k}}{k\,\vartheta_{q^k}(1)}\,
\sum_{j \!\!\!\!\!\!\!\! \pmod{k}} e^{2n j\pi i/k}\,\vartheta^k(q^{n/k}e^{2j\pi i/k})\,.
\end{equation}
\end{prop}
\begin{proof}
For $r \in \{0,\ldots,k-1\}$ and $j \in \Z$, divide by $x^r$ both members of
\eqref{equation:k} and then replace $x$ by $x\mu^j$; it follows that:
$$
x^{-r}\,\mu^{-rj}\,\vartheta^k(x\mu^j) =
\sum_{n=0}^{k-1}c_{k,n}(q)\,x^{n-r}\,\mu^{(n-r)j}\,\vartheta_{q^k}(x^kq^n)\,.
$$
With $r=n$ and index $j$ running along a complete system of integers modulo $k$,
we find that coefficients $c_{k,n}(q)$ satisfy the following identities: 
$$
x^{-n}\,\sum_{j \!\!\!\!\!\!\!\! \pmod{k}} \mu^{-n j}\,\vartheta^k(x\mu^j) =
k\,c_{k,n}(q)\,\vartheta_{q^k}(x^kq^n)\,.
$$
Substituting $xq^{-n/k}$ for $x$ above, we deduce \eqref{equation:ck1}. \\
Putting $x=1$ and $\mu=e^{-2\pi i/k}$ in \eqref{equation:ck1}, we immediately obtain
\eqref{equation:ck2}.
\end{proof}

Note that, since
$\displaystyle\vartheta(q^{-n/k}\mu^j)=\vartheta(q^{n/k}\mu^{-j})=
q^{1/2-n/k}\,\mu^{j}\vartheta(q^{(n-k)/k}\mu^{-j})$,
one can put \eqref{equation:ck2} under the following form:
$$
c_{k,n}(q)=
\frac{q^{\left((k-n)^2+n^2\right)/(2k)}}{k\,\vartheta_{q^k}(1)}\,\sum_{j=0}^{k-1}\mu^{(k-n) j}\,\vartheta^k(q^{(n-k)/k}\mu^{-j})\,.
$$
Comparing it with formula \eqref{equation:ck2} where one has replaced $n$ by
$k-n$, one recovers the symmetry property of $c_{k,n}$ already encountered.

  
\subsubsection{Expression as a finite sum of theta quotients}

We shall now consider, when $q \to -1$, the asymptotic behaviour of $c_{k,n}(q)$.
We shall for that use the modular transformation of $\vartheta$. So we borrow from
\cite[p. 166, (76.1)]{RademacherTopics} the notation $\vartheta_3(v\vert \tau)$:
$$
\vartheta_3(v\vert\tau) := \displaystyle\sum_{n\in\Z} e^{\pi \ii (n^2\tau+2nv)},
$$
whence
$$
\vartheta_q(x) = \vartheta_3(v\vert\tau) \text{~for~}
x=e^{2\pi \ii v} \text{~and~} q=e^{2\pi \ii \tau}.
$$
Let $M :=\displaystyle \begin{pmatrix} -1 & 0 \cr 2 & -1 \end{pmatrix}$,
$\tau' := M \tau = \displaystyle\frac{-\tau}{2\tau-1}$ and $q' := e^{2\pi \ii \tau'}$.
If $a=-1$, $b=0$, $c=2$ and $d=-1$, note that
$\vartheta_{1-a-c,1-b-d} = \vartheta_{0,2} = \vartheta_3$, where $\vartheta_{\mu,\nu}$
is defined in \cite[p. 181, (81.2)]{RademacherTopics}. From that, the modular formula
given for $\vartheta_3$ in \cite[p. 182, Theorem]{RademacherTopics} implies that:
$$
\vartheta_q(x) =
\dfrac{\ii}{\varepsilon_2}\,e^{-2\pi \ii vv'}\,\sqrt{\dfrac{\ii}{2\tau-1}}\,\vartheta_{q'}(x')\,,
$$
where
$\varepsilon_2 := \displaystyle e^{\pi \ii/4}\,\left(\frac{2}{-1}\right)$,
$x := e^{2\pi \ii v}$, $v'=\displaystyle\frac v{2\tau-1}$ and $x'=e^{2\pi \ii v'}$.
Since $\displaystyle\left(\frac{2}{-1}\right) = \left(\frac{2}{1}\right)=1$, one deduces that:
\begin{equation}
\label{equation:vartheta_3}
\vartheta_q(x)=e^{2\pi \ii (1/8-vv')}\,\sqrt{\frac{\ii}{2\tau-1}}\,\vartheta_{q'}(x')\,.
\end{equation}

Recall that $\gamma_{k,n}(q) = c_{k,n}(q)\,q^{-n/2}$.

\begin{prop}
\label{prop:tkl-1}
The following formula holds:
\begin{equation}
\label{equation:ckvartheta}
\gamma_{k,n}(q) =
\dfrac{e^{\pi \ii (k-2n)/4}\,(qe^{-\pi \ii})^{n(n-k)/(2k)}}
      {k\,\vartheta_{q^k}(1)}\,\left(\frac{\ii}{2\tau-1}\right)^{k/2}\,S_{k,n}(q)\,,
\end{equation}
where
\begin{equation}
\label{equation:ckvarthetaS}
S_{k,n}(q) := \sum_{j \!\!\!\!\!\! \pmod{k}}
\left(q'e^{\pi \ii}\right)^{(2j+n)^2/(2k)}\,\vartheta_{q'}^k\left(q'^{(2j+n)/k}e^{2\pi \ii j/k}\right)\,.
\end{equation}
\end{prop}
\begin{proof}
Note that $\tau = \displaystyle\frac{\tau'}{2\tau'+1}$ and $(2\tau-1)(2\tau'+1)=-1$.
For $\displaystyle x := q^{n/k} e^{2j\pi \ii/k}$, write $v :=\displaystyle\frac{n\tau+j}k$
so that:
$$
v = \dfrac{n}{2k}(2\tau-1)+\frac{2j+n}{2k}\,,\quad 
v' = \dfrac{n}{2k}+\frac{2j+n}{2k(2\tau-1)} = -\dfrac{j}{k}-\frac{2j+n}{k}\tau'
$$
and
$$
v\,v' =
\dfrac{n^2}{4k^2}(2\tau-1)+\frac{n(2j+n)}{2k^2} +
\dfrac{(2j+n)^2}{4k^2(2\tau-1)} =
\dfrac{n^2\,\tau}{2k^2} - \dfrac{j^2}{k^2} - \dfrac{(2j+n)^2\tau'}{2k^2}\,\cdot
$$
Applying \eqref{equation:vartheta_3} to $\displaystyle x = q^{n/k}e^{2j\pi i/k}$ and
considering relation $\vartheta_{q'}(x') = \vartheta_{q'}(1/x')$, one gets:
\begin{equation}
\label{equation:vartheta_31}
\vartheta_q^k(q^{n/k}e^{2j\pi \ii/k}) =
C_j\,\left(\frac{i}{2\tau-1}\right)^{k/2}\,q'^{(2j+n)^2/(2k)}\,\vartheta_{q'}^k
\left(q'^{(2j+n)/k}e^{2\pi ij/k}\right)\,,
\end{equation}
where $C_j := e^{2\pi \ii\left(k/8-(n^2\tau-2j^2)/(2k)\right)}$. Since
$$
e^{2n j\pi \ii/k}\,C_j = e^{2\pi \ii\left(k/8-(n^2-(2j+n)^2)/(4k)\right)}q^{-n^2/(2k)}\,,
$$
the relations \eqref{equation:ck2} and \eqref{equation:vartheta_31} imply:
\begin{equation*}
\begin{split}
c_{k,n}(q) =&
\dfrac{e^{\pi ik/4}\,(qe^{-\pi \ii})^{n^2/(2k)}}{k\,\vartheta_{q^k}(1)}\,
\left(\dfrac{\ii}{2\tau-1}\right)^{k/2}\,\times\cr
& \sum_{j \!\!\!\!\!\! \pmod{k}} \left(q'e^{\pi \ii}\right)^{(2j+n)^2/(2k)}\,\vartheta_{q'}^k
\left(q'^{(2j+n)/k}e^{2\pi \ii j/k}\right)\,,
\end{split} 
\end{equation*}
which is plainly equivalent to \eqref{equation:ckvartheta}.
\end{proof}


\subsubsection{Sufficient conditions for vanishing at real negative values}

The main goal of all the rest of this section is the proof of the following result.

\begin{thm}
\label{theo:limite}
Let $k$ an integer $\ge 3$ and $n$ an integer lying between $0$ and $k-1$, and
suppose that $q$ tends to $-1$ within the real interval $]-1,0[$. 
\begin{itemize}
\item One has $\displaystyle\lim_{q\to-1}\gamma_{k,n}(q)=0$ if, and only if, $k$
is even and $k - 2n \equiv 2 \pmod 4$.
\item Otherwise,
$\displaystyle\lim_{q\to-1}\left\vert\gamma_{k,n}(q)\right\vert = + \infty$.
More precisely, one has
$\displaystyle\lim_{q\to-1}(-1)^{k'-n'}\,\gamma_{k,n}(q) = + \infty$ if $(k,n)$
can be expressed by means of a pair of integers $(k',n')$ in one of the following manners:
\begin{enumerate}
 \item $k=4k'$, $n=2n'$;
 \item $k=4k'+2$, $n=2n'+1$;
 \item $k=4k'+1$, $n\in\{2n',2n'+1\}$;
 \item $k=4k'-1$, $n\in\{2n',2n'-1\}$.
\end{enumerate}
\end{itemize}
\end{thm}

Proof of theorem \ref{theo:limite} comes through a series of lemmas and remarks. We now
give a direct consequence on zeroes of $\gamma_{k,n}(q)$. Assume that $k \ge 3$ and
$n$ lies between $0$ and $k-1$. Applying theorem \ref{theo:limite}, we find that
$\gamma_{k,n}(q)\to-\infty$ for
$q\to -1$ only in the following cases: \\
(i) $k \equiv 0$ or $4 \pmod 8$ and $n \equiv 2$ or $0 \pmod 4$ respectively; \\
(ii) $k \equiv 2$ or $6 \pmod 8$ and $n \equiv 3$ or $1 \pmod 4$ respectively; \\
(iii) $k \equiv 1$ or $5 \pmod 8$ and $n \equiv 2$, $3$ or $0$, $1 \pmod 4$
respectively; \\
(iv) $k \equiv 3$ or $7 \pmod 8$ and $n \equiv 0$, $3$ or $2$, $1 \pmod 4$
respectively. \\

In all these cases $\gamma_{k,n}(q)$ admits a zero somewhere over $\left]-1,0\right[$.
From that point on, we easily recover corollary \ref{cor:CSAnnulation} and its consequences.


\subsubsection{Application of modular forms to $\vartheta_{q^k}(1)$}

Suppose that $\tau \to \displaystyle \frac 1 2$ and $q \to -1$. More precisely,
write $\tau = \displaystyle \frac 1 2 + \ii t$ with $t > 0$; then:
\begin{equation}
\label{equation:ttau}
q = e^{\pi \ii - 2 \pi t},\quad  \tau' =
\displaystyle -\frac 1 2 + \frac{\ii}{4t}\,,\quad q'\displaystyle=e^{-\pi \ii-{\pi}/{2t}}\,.
\end{equation} 
It follows that $q'$ will tend to $0$ exponentally fast as $t \to 0^+$. For the sake
of simplicity shall denote $\varnothing$ any function  $f$ defined over $]0,+\infty[$
and such that $f(t) = O(e^{-\kappa/t})$ for $t \to 0$ for some $\kappa > 0$. This being
said, taking $x := 1$ in \eqref{equation:vartheta_3}, one gets:
\begin{equation}
\label{equation:theta1}
\vartheta_q(1) = \frac{e^{\pi \ii/4}}{\sqrt{2t}}\,\left(1+\varnothing\right)\,.
\end{equation}
This will be generalized as follows.

\begin{lem}
Let $k \in \N^*$. If $q = e^{\pi \ii-2\pi t}$ with $t \to 0^+$, we have:
\begin{equation}
\label{equation:thetak1}
\vartheta_{q^k}(1)=\left\{
 \begin{array}{ll}\displaystyle
  \frac1{\sqrt{kt}}\,(1+\varnothing)&k\equiv 0\pmod 4;\cr\cr\displaystyle
  \frac{e^{\pi \ii/4}}{{\sqrt{2kt}}}\,(1+\varnothing)&k\equiv 1\pmod 4;\cr\cr\displaystyle
  \frac2{\sqrt{kt}}\,e^{-1/(4kt)}\,(1+\varnothing)&k\equiv 2\pmod 4;\cr\cr\displaystyle
  \frac{e^{-\pi \ii/4}}{{\sqrt{2kt}}}\,(1+\varnothing)&k\equiv 3\pmod 4.\cr
\end{array}
\right.
\end{equation} 
\end{lem}
\begin{proof}
We first write
$\vartheta_{q^k}(1)=
\vartheta_3(0\vert k\tau)=\vartheta_3(0\vert \displaystyle \frac k2+{kt}\ii)$.
Taking in account formula \cite[p.181, (81.12)]{RademacherTopics}, we obtain 
the following expressions:
\begin{equation}
\label{equation:thetak1a}
  \vartheta_{q^k}(1)=\left\{
 \begin{array}{ll}\displaystyle
  \vartheta_3(0\,\vert\, kt\ii)&k\equiv 0\pmod 4;\cr\cr\displaystyle
\vartheta_3(0\,\vert\, \frac12+kt\ii)&k\equiv 1\pmod 4;\cr\cr\displaystyle
  \vartheta_4(0\,\vert \,kt\ii)&k\equiv 2\pmod 4;\cr\cr\displaystyle
  \vartheta_4(0\,\vert\, \frac12+ kt\ii)&k\equiv 3\pmod 4.\cr
 \end{array}
\right.
\end{equation} 
If $k \equiv 0 \pmod 4$, looking at the first relation of \eqref{equation:thetak1a}
and the second modular relation from \cite[p. 177, (79.9)]{RademacherTopics} implies
that 
$$
\vartheta_3(0\,\vert\,kt\ii) =
\dfrac1{\sqrt{kt}}\,\vartheta_3(0\,\vert\,\frac{\ii}{kt})
= \dfrac1{\sqrt{kt}}\,\sum_{n\in\Z}e^{-\pi n^2/(kt)}\,,
$$
whence the first asymptotic estimation stated in \eqref{equation:thetak1}. \\
If $k\equiv 1\pmod 4$, one obtains the corresponding asymptotic estimate
\eqref{equation:thetak1} from replacing $t$ by $k t$ in \eqref{equation:theta1}. \\
If $k\equiv 2\pmod 4$, consider the third relation in \eqref{equation:thetak1a} and
the first modular relation in \cite[p. 177, (79.9)]{RademacherTopics}: this implies
that
$$
\vartheta_4(0\,\vert\,kt\ii)=
\frac1{\sqrt{kt}}\,\vartheta_2(0\,\vert\,\frac{i}{kt})
= \frac1{\sqrt{kt}}\,\sum_{n\in\Z}e^{-\pi (n+1/2)^2/(kt)}\,,
$$
where $\vartheta_2$ is defined in \cite[p.166, (76.1)]{RademacherTopics}. One deduces:
$$
\vartheta_4(0\,\vert\,kt\ii)
=\dfrac1{\sqrt{kt}}\,\left(2e^{-\pi/(4kt)} + 2 e^{-9\pi/(4kt)} + \cdots \right) =
\dfrac2{\sqrt{kt}}\,e^{-\pi/(4kt)}\,(1+\varnothing)\,,
$$
which yields the third formula in \eqref{equation:thetak1}. \\
Last, if $k\equiv 3\pmod 4$, we shall use the last relation of \eqref{equation:thetak1a}.
Applying the last modular formula from \cite[p. 182, Theorem]{RademacherTopics} to the
quadruple $(a,b,c,d)=(-1,0,2,-1)$, we see that, for
$\tau'=\displaystyle\frac{-\tau}{2\tau-1}$:
$$
\vartheta_4(v'\,\vert\,\tau')=\ii^2\,\epsilon_2\,\sqrt{\frac{2\tau-1}\ii}\,e^{2\pi \ii vv'}\,\vartheta_{2,1}(v\,\vert\,\tau)\,,
$$
where, as before, $\epsilon_2=e^{\pi \ii/4}$ and $v$, $v'$ are the same as in
\eqref{equation:vartheta_3}. Modifying the first relation from
\cite[p. 181, (81.31)]{RademacherTopics} into the form
$\vartheta_{\mu+2,\nu}(v\,\vert\,\tau)=(-1)^\nu\,\vartheta_{\mu,\nu}(v\,\vert\,\tau)$,
we find that
$\vartheta_{2,1}(v\,\vert\,\tau)=
-\vartheta_{0,1}(v\,\vert\,\tau)=
-\vartheta_4(v\,\vert\,\tau)$.
Putting $v=0$ and $\tau=\displaystyle\frac12+t\ii$ hereabove, one draws that 
$$
\vartheta_4(0\,\vert\,\frac12+t\ii)=
\dfrac{e^{-\pi \ii/4}}{\sqrt{2t}}\,\vartheta_{4}(0\,\vert\,-\dfrac{1}{2}+\dfrac{\ii}{t})\, \cdot
$$
Expanding $\vartheta_4(0\,\vert\,\tau)$ with the help of
\cite[p.166, (76.1)]{RademacherTopics}, we have:
$$
\vartheta_{4}(0\,\vert\,-\frac{1}{2}+\frac{\ii}{t})=
\sum_{n\in\Z}(-1)^n\,e^{-\pi \ii n^2/2-\pi n^2/t}=
1+\varnothing\,.
$$
Substituting $kt$ for $t$ in the last relation, we eventually obtain the complete
asymptotic estimation \eqref{equation:thetak1}.
\end{proof}

Combining \eqref{equation:thetak1} with \eqref{equation:ckvartheta}, we deduce the
following remark.

\begin{rem}
\label{rem:gamma}
Let $k$ and $n$ integers such that $0\le n<k$ and let $\Gamma_{k,n}$ the
function defined over $]0,+\infty[$ by the relation
\begin{equation}
\label{equation:Gamma}
\Gamma_{k,n}(t) :={k^{-1/2}\,2^{-k/2}}\,e^{\pi n(n-k)t/k}\,t^{(1-k)/2}\,.
\end{equation}
If $q=e^{\pi \ii-2\pi t}$ with $t\to 0^+$, one has:
\begin{equation}
 \label{equation:ckvartheta1}
\gamma_{k,n}(q)=\left\{
 \begin{array}{ll}\displaystyle
  e^{\pi \ii(k-2n)/4}\,\Gamma_{k,n}(t)\,S_{k,n}(q)\,(1+\varnothing)&k\equiv 0\pmod 4;\cr\cr\displaystyle
  \sqrt2\,e^{\pi \ii(k-1-2n)/4}\,\Gamma_{k,n}(t)\,S_{k,n}(q)\,(1+\varnothing)&k\equiv 1\pmod 4;\cr\cr\displaystyle
  \frac{1}{2}\,e^{\pi \ii(k-2n)/4}\,\Gamma_{k,n}(t)\,S_{k,n}(q)\,e^{\pi /(4kt)}\,(1+\varnothing)&k\equiv 2\pmod 4;\cr\cr\displaystyle
  \sqrt2\,e^{\pi \ii(k+1-2n)/4}\,\Gamma_{k,n}(t)\,S_{k,n}(q)\,(1+\varnothing)&k\equiv 3\pmod 4.\cr
 \end{array}
\right.
\end{equation}
\end{rem}


\subsubsection{Asymptotic study of theta functions with modular variables}

We now are left to study the asymptotic behaviour of $S_{k,n}(q)$ for $q \to -1$.
In order to do so, let us consider expression \eqref{equation:ckvarthetaS}, and put:
\begin{equation}
 \label{equation:Aalpha}
 s_\alpha(q) = s_{k,n,\alpha}(q) =
 \left(q'e^{\pi \ii}\right)^{\alpha^2/2}\,\vartheta_{q'}
 \left((q'e^{\pi \ii})^{\alpha}e^{-\pi \ii n/k}\right)
\end{equation}
for every $\alpha\in\R$. Since $\vartheta_q(x) = 0$ for $x \in -q^{\frac12+\Z}$,
we draw the following remark:

\begin{rem}
If $n/k=1/2$ and $\alpha=-1/2$, one has:
\begin{equation}
\label{equation:Aalpha1/2}
s_{-1/2}(q)=0\,.
\end{equation}
\end{rem}

Expanding $\vartheta_{q'}$ in series with $q'=e^{-\pi i-\pi/(2t)}$ (see \eqref{equation:ttau}), we have:
\begin{equation}
\label{equation:Aalpha1}
s_\alpha(q) = \sum_{m\in\Z} e^{-\pi i m(m/2+n/k)-\pi(\alpha+m)^2/(4t)}\,,
\end{equation}
which implies that $s_{\alpha+2}(q) = s_\alpha(q)$. \\

So let us make up a ``fundamental'' system with $k$ elements as follows: 
$$
\Lambda_{k,n} \equiv \displaystyle\left \{\frac{2j+n}{k}\ \pmod{2}:j\in\Z\right\}\,.
$$ 
More precisely, we will choose $\Lambda_{k,n}$ in the following way:
\begin{equation}
  \label{equation:Lambda}
  \Lambda_{k,n} = \left\{ 
  \begin{array}{ll}
   \displaystyle
   \left \{-1+\frac2k,\ldots,0,\ldots,1\right\}&(k,n) \equiv (0,0)\,\pmod{2};\cr\cr
   \displaystyle
   \left \{-1+\frac1k,\ldots,\frac1k,\ldots,1-\frac1k\right\}&(k,n) \equiv (0,1)\,\pmod{2};\cr\cr
   \displaystyle
   \left \{-1+\frac1k,\ldots,0,\ldots,1-\frac1k\right\}&(k,n) \equiv (1,0)\,\pmod{2};\cr\cr
   \displaystyle
   \left \{-1+\frac2k,\ldots,\frac1k,\ldots,1\right\}&(k,n) \equiv (1,1)\,\pmod{2}.\cr
  \end{array}
  \right. 
\end{equation}
When $\alpha$ runs in $\Lambda_{k,n}$, the terms $s_\alpha$ defined in
\eqref{equation:Aalpha} give rise to the following identity, coming from
\eqref{equation:ckvarthetaS}:
\begin{equation}
\label{equation:SA}
S_{k,n}(q) = \sum_{\alpha\in\Lambda_{k,n}}s_\alpha^k(q)\,.
\end{equation}

Successively taking $\alpha := 0$ and $1$ in expression \eqref{equation:Aalpha1},
one finds that: 
$$
s_0(q) = 1 + 2 \sum_{m=1}^\infty\cos(mn\pi/k)\,e^{-\pi \ii m^2/2-\pi m^2/(4t)}
$$
and:
$$
s_1(q) =
e^{\pi \ii(n/k-1/2)}\,\left(1+2\sum_{m=1}^\infty(-1)^m\cos(mn\pi/k)\,e^{-\pi \ii m^2/2-\pi m^2/(4t)}\right)\,.
$$
One deduces that:
\begin{equation}
 \label{equation:Aalpha0}
 s_0^k(q) = 1 - 2k\ii\,\cos(n\pi/k)\,e^{-\pi/(4t)}+O(e^{-\pi/(2t)})
\end{equation}
and
\begin{equation}
 \label{equation:Aalpha01}
 s_1^k(q) =
 e^{\pi \ii(n-k/2)}\,\left(1+2k\ii\,\cos(n\pi/k)\,e^{-\pi/(4t)}+O(e^{-\pi/(2t)})\right)\,.
\end{equation}

Let $\alpha\in]0,1[$ and put $\alpha' := \min((\alpha-2)^2,(\alpha+1)^2) > 1$; we shall
assume that $k\ge 3$. Using expression  \eqref{equation:Aalpha1}, we obtain that
$$
s_\alpha(q)=
e^{-\pi \alpha^2/(4t)}-\ii\,e^{\pi \ii n/k-\pi(1-\alpha)^2/(4t)}+O(e^{-\pi\alpha'/(4t)})
$$
and
$$
s_{-\alpha}(q)=
e^{-\pi \alpha^2/(4t)}-\ii\,e^{-\pi \ii n/k-\pi(1-\alpha)^2/(4t)}+O(e^{-\pi\alpha'/(4t)})\,.
$$
In particular, if $\alpha=1/k$, then $\alpha'=(k+1)^2/k^2$; taking in account
relations $(1-\alpha)^2-\alpha^2=(k-2)/k$ and $\alpha'-\alpha^2=(k+2)/k$ entails:
$$
s_{1/k}(q)=
e^{-\pi /(4k^2t)}\,\left(1-\ii\,e^{\pi \ii n/k-\pi(k-2)/(4kt)}+O(e^{-\pi(k+2)/(4kt)})\right)
$$
and
$$
s_{-1/k}(q)=e^{-\pi /(4k^2t)}\,
\left(1-\ii\,e^{-\pi \ii n/k-\pi(k-2)/(4kt)}+O(e^{-\pi(k+2)/(4kt)})\right)\,.
$$
If ${k^*}=\min(2k-4,k+2)$, one gets that:
\begin{equation}
 \label{equation:Aalpha+}
 s_{1/k}^k(q)+s_{-1/k}^k(q)=
2\,e^{-\pi/(4kt)}\,\left(1-2k\ii\,\cos(n\pi/k)\,e^{-\pi(k-2)/(4kt)}+O(e^{-\pi{k^*}/(4kt)})\right)\,.
\end{equation}
In a similar way, one will find that: 
\begin{equation}
 \label{equation:Aalpha+k}
 \begin{split}
 s_{1-1/k}^k(q)+s_{-1+1/k}^k(q)=&(-1)^{k+n}\,2\,i^k\,e^{-\pi/(4kt)}\,\times\cr
 &\left(1+2ki\,\cos(n\pi/k)\,e^{-\pi(k-2)/(4kt)}+O(e^{-\pi{k^*}/(4kt)})\right)\,.
 \end{split}
\end{equation}
Moreover, if $k\ge 4$ and $\alpha\in\displaystyle\left[\frac2k,1-\frac2k\right]$:
\begin{equation}
 \label{equation:Aalpha+kk}
 s_{\alpha}^k(q)+s_{-\alpha}^k(q)=O(e^{-\pi/(kt)})\, .
\end{equation}

\begin{prop}
\label{prop:S}
Let $k\ge 3$ and $0\le n<k$. The following asymptotic estimates hold $S_{k,n}(q)$
whenever $q \displaystyle = e^{\pi \ii-2\pi t}$ and $t \to 0^+$:
 \begin{enumerate}
  \item If $(k,n) \equiv (0,0)\,\pmod{2}$:
  \begin{equation}
   \label{equation:00}
   S_{k,n}(q)=2\,e^{\pi \ii(2n-k)/4}\,\cos \frac{(2n-k)\pi}4+O(e^{-\pi/(kt)})\,.
  \end{equation}
\item If $(k,n) \equiv (0,1)\,\pmod{2}$:
  \begin{equation}
    \label{equation:01}
    S_{k,n}(q)=2\,(1-\ii^k)\,e^{-\pi/(4kt)}+O(e^{-\pi\min(k-2,4)/(4kt)})\,.
  \end{equation}
  \item If $(k,n) \equiv (1,0)\,\pmod{2}$:
  \begin{equation}
    \label{equation:10}
    S_{k,n}(q)=1-2\,\ii^k\,e^{-\pi/(4kt)}+O(e^{-\pi\min(k-1,4)/(4kt)})\,.
  \end{equation}
  \item If $(k,n) \equiv (1,1)\,\pmod{2}$:
  \begin{equation}
    \label{equation:11}
    S_{k,n}(q)=\ii^k+2\,e^{-\pi/(4kt)}+O(e^{-\pi\min(k-1,4)/(4kt)})\,.
  \end{equation}
 \end{enumerate}
\end{prop}
\begin{proof}
We shall use the expression of $S_{k,n}(q)$ given in \eqref{equation:SA} with the
various forms of $\Lambda_{k,n}$ shown in \eqref{equation:Lambda}. \\
(1) Suppose that $(k,n) \equiv (0,0)\,\pmod{2}$; one has $k \ge 4$. Because of
the first relation of \eqref{equation:Lambda}, one can write \eqref{equation:SA}
in the following form:
$$
S_{k,n}(q) = s_0^k(q)+s_1^k(q)+\sum_{j=1}^{\frac k2-1}\left(s_{2j/k}^k(q)+s_{-2j/k}^k(q)\right)
$$
Consider \eqref{equation:Aalpha0}, \eqref{equation:Aalpha01} together with
\eqref{equation:Aalpha+kk}, and note that $\frac 1 2 \ge \frac 2 k$; then:
$$
S_{k,n}(q) = 1 + e^{\pi \ii(n-k/2)}+O(e^{-\pi/(kt)})\,,
$$
which takes us to relation \eqref{equation:00}. \\
(2) Suppose that $(k,n) \equiv (0,1)\,\pmod{2}$; again, $k\ge 4$. Thanks to the second relation
of \eqref{equation:Lambda}, expression \eqref{equation:SA} can be put in the following
form:
\begin{equation*}
 \begin{split}
 S_{k,n}(q)=&\left(s^k_{1/k}(q)+s_{-1/k}^k(q)\right)+\left(s_{1-1/k}^k(q)+s_{1/k-1}^k(q)\right)\cr
 &+\sum_{j=1}^{\frac k2-2}\left(s_{(2j+1)/k}^k(q)+s_{-(2j+1)/k}^k(q)\right)\,.
\end{split}
\end{equation*}
Since $(-1)^{k+n}\,\ii^k=-\ii^k$, one obtains directly \eqref{equation:01}, taking
in account relations \eqref{equation:Aalpha+}, \eqref{equation:Aalpha+k} together with
\eqref{equation:Aalpha+kk}. \\
(3) Suppose that $(k,n) \equiv (1,0)\,\pmod{2}$. By virtue of the third relation of
\eqref{equation:Lambda}, one transforms \eqref{equation:SA} into:
\begin{equation*}
 \begin{split}
 S_{k,n}(q)=&s^k_0(q)+\left(s_{1-1/k}^k(q)+s_{1/k-1}^k(q)\right)+\sum_{j=1}^{\frac {k-3}2}\left(s_{2j/k}^k(q)+s_{-2j/k}^k(q)\right)\,.
\end{split}
\end{equation*}
Relation \eqref{equation:10} flows immediately from the estimations contained in
\eqref{equation:Aalpha0}, \eqref{equation:Aalpha+k} combined with
\eqref{equation:Aalpha+kk}. \\
(4) Last, suppose that $(k,n) \equiv (1,1)\,\pmod{2}$. By virtue of the last relation of
\eqref{equation:Lambda}, one expresses \eqref{equation:SA} in the following form:
\begin{equation*}
 \begin{split}
 S_{k,n}(q)=&s^k_1(q)+\left(s_{1/k}^k(q)+s_{-1/k}^k(q)\right)+\sum_{j=1}^{\frac {k-3}2}\left(s_{(2j+1)/k}^k(q)+s_{-(2j+1)/k}^k(q)\right)\,.
\end{split}
\end{equation*}
Meanwhile, also note that $e^{\pi \ii(n-k/2)}=(-1)^{n-k}\,\ii^k=\ii^k$. By way of
consequence, relation \eqref{equation:11} directly follows from formulas
\eqref{equation:Aalpha01}, \eqref{equation:Aalpha+} with \eqref{equation:Aalpha+kk}.
\end{proof}


\subsubsection{End of the proof of theorem \ref{theo:limite}}

Considering remark \ref{rem:gamma} together with proposition \ref{prop:S}, we obtain
the following result.

\begin{thm}
\label{theo:gamma}
Let $k\ge 3$ and $0\le n<k$ and let $\Gamma_{k,n}$ the function defined in
\eqref{equation:Gamma}. We have the following asymptotic estimations for
$\gamma_{k,n}(q)$ when $q\displaystyle=e^{\pi i-2\pi t}$ and $t\to0^+$:
\begin{enumerate}
\item If $k$ is even and $k- 2n\equiv2\pmod 4$, then
$\gamma_{k,n}(q)=\varnothing$.
\item If $k=4k'$ and $n=2n'$ or if $k=4k'+2$ and $n=2n'+1$, where $k'$, $n'\in\Z$, then:
\begin{equation}
\label{equation:g0}
\gamma_{k,n}(q)=     (-1)^{k'-n'}\,2\,\Gamma_{k,n}(t)\,\left(1+\varnothing\right)\,.
\end{equation}
\item If $k=4k'+1$ and $n\in\{2n',2n'+1\}$ or if $k=4k'-1$ and
$n\in\{2n',2n'-1\}$, where $k'$, $n'\in\Z$, then:
\begin{equation}
\label{equation:g1}\gamma_{k,n}(q)=
(-1)^{k'-n'}\,\sqrt2\, \Gamma_{k,n}(t)\,\left(1+\varnothing\right)\,.
\end{equation}
\end{enumerate}
\end{thm}

\begin{proof}
Suppose first that $k\equiv0 \pmod 4$ and $n\equiv 1 \pmod 2$. As $k\ge 4$, relation  \eqref{equation:01}
shows that $S_{k,n}(q)=O(e^{-\pi\min(k-2,4)/(4kt)})=O(e^{-\pi/(2kt)})$; one deduces with the help of first and third
relations of \eqref{equation:ckvartheta1}, that
$\gamma_{k,n}(q)=O(\Gamma_{k,n}(q)\,e^{-\pi/(4kt)})=\varnothing$. Moreover, if
$k\equiv2 \pmod 4$ and $n\equiv 0 \pmod 2$, relation \eqref{equation:00} combined with the third formula of
\eqref{equation:ckvartheta1} implies that
$\gamma_{k,n}=O(\Gamma_{k,n}(q)\,e^{-\pi/(4kt)})=\varnothing$. Thus the first
statement of the theorem is obtained. \\
If $k=4k'$ and $n=2n'$ with $k'$, $n'\in\Z$, one simultaneously considers
the first relation of \eqref{equation:ckvartheta1} and formula \eqref{equation:00}
and one finds that 
$$
\gamma_{k,n}(q)=2\Gamma_{k,n}\cos(n'-k')\pi+\varnothing,
$$
which is equivalent to \eqref{equation:g0}. The case of $k=4k'+2$ and $n=2n'+1$ may be treated in a similar way, using the third relation of \eqref{equation:ckvartheta1} and formula \eqref{equation:01}. \\
In the case $k=4k'+1$ and $n\in\{2n',2n'+1\}$ one concludes directly by
combining the second relation of \eqref{equation:ckvartheta1} with \eqref{equation:10}
or \eqref{equation:11}. \\ 
As for the remaining cases, if $k=4k'-1$ and $n=2n'$, relation \eqref{equation:10}
and the last formula of \eqref{equation:ckvartheta1} imply \eqref{equation:g1}. Last,
if $k=4k'-1$ and $n=2n'-1$, relation \eqref{equation:11} and the last formula of
\eqref{equation:ckvartheta1} also imply \eqref{equation:g1}.
\end{proof}

\begin{proof} [Proof of theorem \ref{theo:limite}]
It follows directly from theorem \ref{theo:gamma}.
\end{proof}


\bibliography{PowersTheta}

\end{document}